\documentclass[12pt,a4paper]{amsart}
\usepackage{geometry}
\usepackage[all]{xy}
\usepackage{graphicx}
\usepackage{amssymb, amsthm}
\usepackage{epstopdf}
\def\S{\Sigma}

\def\R{\mathbb{R}}
\def\Z{\mathbb{Z}}

\def\C{\mathbb{C}}
\def\G{\Gamma}

\def\<{\langle}
\def\>{\rangle}
\def\e{\varepsilon}
\def\a{\alpha}

\def\he{{\rm{HE}}}

\def\out{{\rm{Out}}(F_n)}
\def\aut{{\rm{Aut}}(F_n)}

\def\outm{{\rm{Out}}(F_m)}
\def\autm{{\rm{Aut}}(F_m)}

\def\Autn{{\rm{Aut}}(F_n)}

\def\dd{{\rm{deck}}}

\newcommand{\cyc}{\theta} 
\newcommand{\Cyc}{\Theta_n} 
\newcommand{\hCyc}{\widehat\Theta_n}
\newcommand{\Oh}{\bf 0} 
\newcommand{\Out}{{\rm{Out}}(F_n)}
\newcommand{\deck}{{\rm Deck}} 

\def\he{{\rm{he}}}
\def\HE{{\rm{HE}}}
\def\fhe{{\rm{fhe}}}
\def\FHE{{\rm{FHE}}}

\newtheorem{thmA}{Theorem}

\newtheorem{corA}{Corollary}

\newtheorem{theorem}{Theorem}[section]

\newtheorem{lemma}[theorem]{Lemma}
\newtheorem{proposition}[theorem]{Proposition}

\newtheorem{corollary}[theorem]{Corollary}
\newtheorem{definition}[theorem]{Definition}

\newtheorem{remark}[theorem]{Remark}
\newtheorem{example}[theorem]{Example}
\newtheorem*{notation}{Notation}

\begin{document}

\catcode`\@=11
\def\serieslogo@{\relax}
\def\@setcopyright{\relax}
\catcode`\@=12

\def\jump{\vskip 0.3cm}
\title[Maps ${\rm{Out}}(F_n)\to {\rm{Out}}(F_m)$]{Abelian covers of
graphs and maps between outer automorphism groups
of free groups}

\author[Bridson]{Martin R.~Bridson}
\address{Martin R.~Bridson\\
Mathematical Institute \\
24-29 St Giles' \\
Oxford OX1 3LB \\
U.K. }
\email{bridson@maths.ox.ac.uk}

\author[Vogtmann]{Karen Vogtmann }
\address{Karen Vogtmann\\
Department of Mathematics\\
Cornell University\\
Ithaca NY 14853 }
\email{vogtmann@math.cornell.edu}

\thanks{Bridson's work was supported by an EPSRC Senior Fellowship.
Vogtmann is supported NSF grant DMS-0204185.}




\begin{abstract} We explore the existence of homomorphisms
between outer automorphism groups of free groups $\out\to\outm$.
We prove that if $n>8$ is even and $n\neq m\le 2n$,
or $n$ is odd and $n\neq m\le 2n-2$, then all such homomorphisms
have finite image; in fact they factor
through ${\rm{det}}\colon \out\to \Z/2$.
In contrast, if $m=r^n(n-1)+1$ with $r$ coprime to
$(n-1)$,
then there exists an embedding $\out\hookrightarrow\outm$. In order
to prove this last statement, we
determine when the action of $\out$ by homotopy equivalences
on a graph of genus $n$ can be lifted to an action on a normal covering with
abelian Galois group.
\end{abstract}
\maketitle

\section{Introduction}
The contemporary study of mapping class groups and outer automophism
groups of free groups is heavily influenced by the analogy between
these groups  and lattices in semisimple
Lie groups. In previous papers \cite{BV1, BV2, BV3} we have explored rigidity properties of $\out$ in this light, proving in particular that
if $m<n$  then any homomorphism $\out\to\outm$  has image  at most $\Z_2$, and that the only monomorphisms  $\out\to\out$ are the inner automorphisms.  
In this paper we turn our attention to the case $m>n$.

There are two obvious ways in which one might embed $\aut$ in $\autm$
when $m>n$: most obviously, the inclusion $F_n\subset F_m$ of any free factor induces a monomorphism $\aut\hookrightarrow\autm$; secondly, if
$N\subset F_n$ is a characteristic subgroup of finite index,
then the restriction map $\aut\to{\rm{Aut}}(N)\cong\autm$ is injective (Lemma~\ref{l:inj}).
Neither of these constructions sends the group of inner
automorphisms ${\rm{Inn}}(F_n)\subset\aut$ into ${\rm{Inn}}(F_m)$,
so there is no induced map $\out\hookrightarrow\outm$. In the second case one
can often remedy this problem by passing to a subgroup of finite index in $\out$.
Thus in Proposition \ref{p:finiteInd} we prove that if $m=d(n-1)+1$ for some $d\ge 1$, then $\out$ 
has a subgroup of finite index that embeds in  $\outm$; for example
a finite-index subgroup of $\out$ embeds in ${\rm{Out}}(F_{2n-1})$.
But if we demand that our homomorphisms be defined on the whole of $\out$,
then it is far from obvious that there are {\em{any}} maps
$\out\to\outm$ with infinite image when $m>n\ge 3$.

As usual,
the case $n=2$ is exceptional: ${\rm{Out}}(F_2)
={\rm{GL}}(2,\Z)$ maps to $(\Z/2)\ast(\Z/3)$ with finite kernel,
so to obtain a map with infinite image one need only choose
elements of order $2$ and $3$ that generate an infinite subgroup of $\outm$.
Khramtsov \cite{Kh} gives an explicit monomomorphism ${\rm{Out}}( F_2)\to {\rm{Out}}( F_4)$. More interestingly, he proved that there are no
injective maps from $\out$ to ${\rm{Out}}(F_{n+1})$. {\em So, for given
$n$, for which values of $m$ is there a monomorphism
$\out\to{\rm{Out}}(F_m)$, and for which values of $m$ do all
maps $\out\to{\rm{Out}}(F_m)$ have finite image?} These are the
questions that we address in this article. In the first part of the
paper we give explicit constructions of embeddings, and in the
second half we prove, among other things,
that no homomorphism $\out\to\outm$ can have image bigger than $\Z/2$
if $n$ is even and $8<n<m\le 2n$. This last result disproves a
conjecture of Bogopolski and Puga \cite{BP}.

In  order to construct embeddings,
we consider characteristic subgroups $N<F_n$, identify $F_n$
with the subgroup of $\aut$ consisting of inner automorphisms, and
examine the short exact sequence
$$1\to F_n/N\to
\aut/N\to \out\to 1.$$
We want to understand when this sequence splits. When it does
split, one can compose the splitting map $\out\to\aut/N$ with
the map $\aut/N\to {\rm{Out}}(N)$ induced by restriction, $\phi\to [\phi|_N]$,
to obtain an embedding of $\out$ into ${\rm{Out}}(N)$.

Bogopolski and Puga   \cite{BP} used algebraic methods to
obtain a splitting in the case where $F_n/N\cong (\Z/r)^n$ with
$r$ odd and coprime to $(n-1)$, yielding
embeddings $\out\hookrightarrow{\rm{Out}}( F_m)$ when $m=r^n(n-1)+1$.
We do not follow their arguments.  Instead we adopt a geometric approach which begins with
a translation of the above splitting problem into a lifting
problem for groups of homotopy equivalences of graphs.
Proposition \ref{p:sameProb} provides a precise formulation of
this translation. (The topological background to it is difficult to
pin down in the literature, so we explain it in detail
in an appendix.)

The following theorem is the main result in the first half
of this paper.

\begin{thmA}\label{t:main}
Let $\widehat X\to X$ be a normal
covering of a
connected graph of genus $n\ge 2$ with abelian Galois
group $A$.
The action of $\out$ by homotopy equivalences on $X$ lifts to
an action by fiber-preserving homotopy equivalences on $\widehat X$ if and only if
$A\cong(\Z/r)^n$ with $r$  coprime to $n-1$.
\end{thmA}

When translated back into algebra,
this theorem is equivalent to the statement that if a characteristic subgroup
$N<F_n$ contains the commutator subgroup $F_n'=[F_n,F_n]$, then
the short exact sequence $1\to F_n/N \to \aut/N \to\out \to 1$ splits if and only if $N=F_n'F_n^r$, where $F_n^r$ is the subgroup
generated by $r$-th powers and $r>1$ is coprime to $n-1$.
The sufficiency of this condition extends
Bogopolski and Puga's theorem  to cover the case
where $r$ is even.

\begin{corA}\label{c:embed} There exists an embedding
$\out\hookrightarrow \outm$ for any $m$ of the
form $m=r^n(n-1)+1$ with $r>1$ coprime to $n-1$.
\end{corA}

The negative part of Theorem \ref{t:main} also has an intriguing
application. It tells us  that
$1\to F_n/F_n' \to \aut/F_n' \to\out \to 1$
does not split. Thus this sequence defines a non-zero class
in the second cohomology group of $\out$ with coefficients
in  the module $M:=F_n/F_n'$ (i.e.~the standard left
$\out$-module  $H_1(F_n)$).  The theorem
also assures us that this class remains
non-trivial when we take coefficients in $M/rM$,
provided that  $r$ is not coprime to $(n-1)$.
The non-triviality of these classes provides a striking counterpoint to what happens
when one takes coefficients in the dual module $M^*=H^1(F_n)$,
as we shall explain in Section \ref{s:Mdual}.

\begin{thmA}\label{t:cohomology} Let $M=H_1(F_n)$ be the standard $\out$-module and let
$M^*$ be its dual. Then $H^2(\out, M)\neq 0$,
but $H^2(\out, M^*)=0$ if $n\ge 12$.
\end{thmA}

Theorem \ref{t:main} exhausts the ways in which one might obtain
embeddings  $\out\to\outm$ by lifting the action of $\out$ to
covering spaces with an abelian Galois group, but one might hope
to construct many other embeddings using non-abelian covers.
Indeed the construction developed
by
Aramayona, Leininger and Souto in the context of surface automorphisms
\cite{ALS} proceeds along exactly these lines and, as they
remark, it can be adapted to the setting of $\out$. However, in the
embeddings $\out\to\outm$ obtained by their method, $m$ is bounded below
by a doubly exponential function of $n$, whereas in our construction
we can take $m=2^n(n-1)+1$ if $n$ is even. If $n$ is odd, then the
smallest value we obtain is $m=p^n(n-1)+1$ where $p$ is the smallest 
prime that does not divide $(n-1)$; in Section \ref{ss:noThy} we describe how
quickly $p$ grows as a function of $n$.

In the second part of this paper we set about the task of providing
lower bounds on the value of $m$ such that there is a monomorphism
$\out\to\outm$, or even a map with infinite image.

\begin{thmA}\label{t:nomaps}  Suppose $n> 8$.
If $n$ is even and $n<m\le 2n$, or $n$ is odd and $n<m\le 2n-2$,
then every  homomorphism
${\rm{Out}}(F_{n})\to{\rm{Out}}(F_{m})$ factors through
${\rm{det}}\colon\out\to\Z/2$.
\end{thmA}

Note how this result contrasts with our earlier
observation that $\out$ has a subgroup of finite index that
embeds in $\outm$ when $m=2n-1$. The key point here is that
subgroups of finite index can avoid certain of the finite subgroups
in $\out$ (indeed they may be torsion-free), whereas our proof of
Theorem \ref{t:nomaps} relies on a detailed understanding of how the
finite subgroups of $\out$ can map to $\outm$ under putative
maps $\out\to\outm$. Two subgroups play a particularly important
role, namely $W_n\cong (\Z/2)^n\rtimes S_n$, the group of symmetries of
the $n$-rose $R_n$, and $G_n\cong\S_{n+1}\times\Z/2$, the group of
symmetries of the {\em{$(n+1)$-cage}}, i.e. the graph with $2$ vertices
and $(n+1)$-edges. Indeed the key idea in the proof of Theorem C is to show that
no homomorphism can restrict to an injection on both of these subgroups.
In order to establish this, we have to analyze in detail all of the
ways in which these finite groups can act by automorphisms on graphs of genus
at most $2n$. In the light of the realization theorem for finite
subgroups of $\outm$, this analysis amounts to a complete description
of the conjugacy classes of the finite subgroups in $\outm$ that are
isomorphic to $A_n,\ W_n$ and $G_n$ (cf.~Propositions \ref{p:An}, \ref{p:Wn} and \ref{p:Gn}). We believe that these results are of independent interest.

Beyond $m=2n$, the analysis of ${\rm{Hom}}(W_n,\outm)$ and
${\rm{Hom}}(G_n,\outm)$ becomes more complex, but several crucial
facts extend well beyond this range (e.g.~Lemma \ref{sizen} and
Proposition \ref{p:image}). Moreover, Dawid Kielak \cite{kielak} has recently
extended our methods to improve the bound $m\le 2n$.
Thus, at the time of writing, we have no good reason to suppose
that the lower bound that Theorem \ref{t:nomaps} imposes on the least
$m>n$ with $\out\hookrightarrow\outm$ is any closer to the truth that
the exponential upper bound provided by Theorem \ref{t:main}.

We thank  Roger Heath-Brown,  Dawid Kielak and Martin Liebeck for their helpful comments.

\section{Theorem~\ref{t:main}: Restatement and Discussion}

In the appendix to this paper we explain in detail the equivalence of various
short exact sequences arising in group theory and topology. In the case of graphs, the basic equivalence
can be expressed as follows.

Let $N$ be a characteristic subgroup of a free group $F,$ let $X$ be a connected graph with fundamental group $F$,
let $p:\widehat X\to X$ be the covering space corresponding to $N$, let $\HE(X)$ be the group of free homotopy classes of homotopy equivalences of $X$,
and let
$\FHE(\widehat X)$ be the group of
 fiber-preserving homotopy classes of fiber-preserving homotopy equivalences of $\widehat X$. Note that the deck transformations of $\widehat X$
lie in the kernel of the natural map $\FHE(\widehat X)\to \HE(X).$

\begin{proposition}\label{p:sameProb}The following diagram of groups is commutative and the vertical maps are isomorphisms:
$$
\begin{matrix}

 1&\to &F_n/N&\to &\aut/N&\to &\out\to 1\\
&&\downarrow&&\downarrow &&\downarrow\\
 1&\to &\deck&\to &\FHE(\widehat X)&\to &\HE(X)\to 1\\

\end{matrix}
$$
\end{proposition}

The characteristic subgroups $N<F_n$ with $F_n/N$ abelian are the commutator subgroup $F_n'=[F_n,F_n]$ and $F_n'F_n^r$, the subgroup
generated by $F_n'$ and all $r$th powers in $F_n$.
By combining this observation with the preceding proposition, we see that Theorem \ref{t:main} is equivalent to the following statement.

\begin{theorem}\label{t:split} Let $F_n$ be a free group of rank $n$ and let $N<F_n$ be a characteristic subgroup with $F_n/N$ abelian.
Then the short exact sequence
$$
1\to F_n/N\to \Autn/N \to \Out \to 1
$$
splits if and only $N= [F_n,F_n] F_n^r$ with $r$ coprime to $n-1$.
\end{theorem}

The existence of splittings is proved  in Section~\ref{s:split} below, and  the non-existence in Section~\ref{s:no-split}.

Any splitting of the
sequence in Theorem \ref{t:split} gives a monomorphism $\Out\hookrightarrow\Autn/N$, which we can compose with the
restriction map $$\Autn/N\to {\rm{Aut}}(N)/N = {\rm{Out}}(N).$$ To complete the proof of Corollary~\ref{c:embed} we need to know that this
last map is injective.  This follows from the  observation below.

\begin{lemma}\label{l:inj} If $F$ is a finitely generated free group and $N<F$ is a characteristic subgroup of finite index, then
the restriction map ${\rm{Aut}}(F)\to{\rm{Aut}}(N)$ is injective.
\end{lemma}
\begin{proof}    If $k$ is the index of $N$ in $F$ and $w$ is an arbitrary element of $F$, then $w^k\in N$.  
If $\phi$ is in the kernel of the restriction map ${\rm{Aut}}(F)\to{\rm{Aut}}(N)$, then $w^k=\phi(w^k)=(\phi(w))^k$.  But elements in $F$ have unique roots, so $w=\phi(w)$ and $\phi$ is the identity.  
\end{proof}

\subsection{Expected value of $m$}\label{ss:noThy}

The subgroup $N=F_n'F_n^r$ has index $r^n$ in $F_n$ so is free of rank $m=r^n(n-1)+1$.  Thus the smallest $m$ for which we obtain an embedding $\out\to\outm$ from Theorem~\ref{t:split} is $m=p^n(n-1)+1$, where $p$ is the smallest prime which does not divide $n-1$.   If $n$ is even we can take $p=2$ but for $n$ odd the size of $p$ as a function of $n$ is not obvious.
However, it turns out that the expected value of $p$ is a constant (which is approximately equal to $3$).  We are indebted to Roger Heath-Brown for the following argument.

For any natural number $k>1$, let $f(k)$ denote the smallest prime number which does not divide $k$ and let $Q(k)$ be the product of all prime numbers strictly less than $k$ (with $Q(2)=1$). An easy consequence of the Prime Number Theorem is that $\log(Q(k))$ is asymptotically equal to $k$. This implies in particular that the infinite series used to define $C$ in the following proposition is convergent.

\begin{proposition}  The expected value $$E(f)=\lim_{x\to\infty} \frac{1}{x}\sum_{k=1}^x f(k),$$
 exists and is equal to the constant $$C:=\sum_p \frac{p-1}{Q(p)},$$ where the sum is over all primes $p$.
\end{proposition}

\begin{proof}
Note that $f(k)=p$ if and only if $Q(p)$ divides $k$ and $p$ does not divide $k$.  The first statement implies,
taking logs, that $\log(Q(p))\leq \log k$, so $p$ can be of order at most
$\log k$. 

By definition,
$$\sum_{k\le x}f(k)=\sum_p p\cdot \#\{k\le x: f(k)=p\}$$
and
\begin{align*}\#\{k\le x: f(k)=p\} & = \#\{k\le x: Q(p)|k\} - \#\{k\le x: pQ(p)|k\}\\
&= \lfloor x/Q(p)\rfloor - \lfloor x/pQ(p)\rfloor\\
&= x\,\frac{p-1}{pQ(p)} +O(1).
\end{align*}
As we just observed, the primes that contribute to the above sum have order at most $\log (x)$, so
\begin{align*}
\frac{1}{x}\sum_{k\le x}f(k) &=  \sum_{p=O(\log x)} \frac{p-1}{Q(p)}
                    + \frac{1}{x}\, O(\sum_{p= O(\log x)} p)\\
                                                 &=  \sum_{p= O(\log x)} \frac{p-1}{Q(p)} + \frac{1}{ x}O(\log^2 x).
 \end{align*}
Letting $x\to\infty$, we get $E(f)=C$.
\end{proof}

Given $n$, the smallest value of $m$ for which Corollary~\ref{c:embed} yields an embedding is $m=f(n-1)^n(n-1)+1$, and the preceding proposition tells us that
``on average" this is no greater than an exponential function of $n$. In the worst case, $m$ can be larger but still only on the order of $e^{n\log\log n}$. Indeed
the worst case arises when $(n-1)=Q(k)$ for some $k$, in which case $k\le f(n-1)<2k$, and since $\log(Q(k))\sim k$  we see that  $f(n-1)$ grows like $\log n$.

\subsection{Embedding a subgroup of finite index}\label{ss:finiteindex}

Corollary~\ref{c:embed} gives conditions under which the entire group $\out$ embeds in $\outm$.
 If we relax this to require only that a subgroup of finite index of $\out$ should embed in $\outm$, we can obtain many more embeddings as follows.

\begin{proposition}\label{p:finiteInd} For all positive integers $n$ and
$d$, there exists a subgroup of finite index $\G\subset \out$ and a
monomorphism
$\G\hookrightarrow \outm$, where $m=d(n-1)+1$.
\end{proposition}

\begin{proof} For $n=1$ the proposition is trivial, and for $n=2$
it follows immediately from the fact that ${\rm{Out}}(F_2)$ has a free
subgroup of finite index.
So we assume that $n\ge 3$ and fix an epimorphism from $F_n$ to a wreath
product $W=G\wr \Z/d$, where $G$ is any finite 2-generator centerless
group ($S_3$ for example). Let $N$
be the kernel of this epimorphism and let $H\supset N$ be the kernel
of the composition $F_n\to W\to \Z/d$.

The set of subgroups in $F_n$ that have the same index as $N$ is finite,
as is the set that have the same index as $H$. The action of
$\aut$ on each of these sets defines a homomorphism to a finite
symmetric group; define $\Gamma_0$ to be the intersection of the
two kernels. Note that $\Gamma_0$ leaves invariant both $H$ and $N$. Let
$\G_1\subset\G_0$ be the kernel of the natural map
$\G_0\to{\rm{Aut}}(F_n/N)$ and note that since the center of $W=F_n/N$
is trivial,
the intersection of $\G_1$ with ${\rm{Inn}}(F_n)=F_n$ is contained in $N$,
and hence in $H$.

Euler characteristic tells us that the rank of the free group $H$ is $d(n-1)+1$. The restriction map $\G_1\to {\rm{Aut}}(H)$,
which is injective as in Lemma \ref{l:inj}, induces an injection
$\G_1/(\G_1\cap H)\hookrightarrow {\rm{Out}}(H)$. To complete the proof,
it suffices to note that
$\G:=\G_1/(\G_{1}\cap H)$ is the image of $\G_1$
in $\out$, since $(\G_1\cap H)=(\G_1\cap F_n)$.
\end{proof}

\begin{remark} The preceding argument shows that if $N$ is the kernel of a map from $F_n$ onto a finite centerless group, then a subgroup of finite index in $\out$  injects into ${\rm Out}(N)$.
\end{remark}

\section{Proof of Theorem~\ref{t:main}: The existence of lifts}\label{s:split}

In order to prove the existence of lifts as asserted in Theorem \ref{t:main} (equivalently
the existence of splittings in Theorem \ref{t:split}), we work with the sequence
$$1\to \deck\to \FHE(\widehat X)\to \HE(X)\to 1$$
where $X=R$ is a 1-vertex graph with $n$ loops (a {\em rose}) and $\widehat X\to X$ is the covering space $L_r\to R$
corresponding to $N<\pi_1X=F_n$, where $N=F_n'F_n^r$ with $r$ coprime to $n-1$.
We work with an explicit presentation of $\out = \HE(R)$.
We take explicit homotopy equivalences of $R$ that generate $\HE(R)$, lift each
to a homotopy equivalence of the universal abelian covering $L$ of $R$, project down to $L_r$,
and prove that the resulting elements of $\FHE(L_r)$ satisfy the defining relations of
our presentation. The case $n=2$ is special:  for $n=2$  one can split $\HE(R)\to \FHE(L)$.

The generators and relations we will use for $\Out$ are based on those given by Gersten   in  \cite{Ge} for ${\rm{SAut}}(F_n)$. We fix a generating set $A=\{a_1,\ldots,a_n\}$ for $F_n$.  Gersten gives an elegant and succinct presentation using generators $\phi_{ab}$ with $a,b\in A\cup A^{-1}, b\neq a, a^{-1}$; here $\phi_{ab}$  corresponds to the automorphism which sends $a\mapsto ab$ and fixes all elements of $A\cup A^{-1}$ other than $a$ and $a^{-1}$.  In Gersten's paper automorphisms act on $F_n$ on the right and the symbol $[\alpha,\beta]$ means $\alpha\beta\alpha^{-1}\beta^{-1}$.  In the current paper we want automorphisms to act on the left  to be consistent with composition of functions in $\HE(R)$, but we would like to use the same commutator convention. Thus for us a Gersten relation of the form $[\alpha,\beta]=\gamma$ becomes $[\beta^{-1},\alpha^{-1}]=\gamma$ or, equivalently, $[\alpha^{-1},\beta^{-1}]=\gamma^{-1}$. His relations, then, are the following:

\begin{itemize}

\item $\phi_{ab^{-1}}=\phi_{ab}^{-1}$

\item  $[\phi^{-1}_{ab},\phi^{-1}_{cd}]=1$ if $a\neq c, d,d^{-1}$ and $b\neq c,c^{-1}$

\item   $[\phi^{-1}_{ab},\phi^{-1}_{bc}]=\phi^{-1}_{ac}$ for  $a\neq c,c^{-1}$

\item $\phi_{ba}\phi_{ab^{-1}}\phi_{b^{-1}a^{-1}}=\phi_{b^{-1}a^{-1}}\phi_{a^{-1}b}\phi_{b a }$

\item  $(\phi_{b^{-1}a^{-1}}\phi_{a^{-1}b}\phi_{ba})^4=1$.
\end{itemize}

We will need to distinguish between right transvections   $\rho_{ij}\colon a_i\mapsto a_ia_j$ and left transvections $\lambda_{ij}\colon  a_i\mapsto a_ja_i$, for $i\neq j$, so we rewrite Gersten's relations using the translation
$\phi_{a_ia_j}=\rho_{ij},$  $\phi_{a_i^{-1}a_j^{-1}}=\lambda_{ij}$, $\phi_{a_ia_j^{-1}}=\rho_{ij}^{-1}$  and  $\phi_{a_i^{-1}a_j}=\lambda_{ij}^{-1}$.   
 
In terms of the $\rho_{ij}$ and $\lambda_{ij}$, Gersten's first relation is unnecessary and the rest of the presentation for ${\rm S}\aut $ becomes
\begin{enumerate}
\item $[\rho_{ij},\rho_{kl}]=[\rho_{ij},\lambda_{kl}]=[\lambda_{ij},\lambda_{kl}]=1$ if $i\neq k,l$ and $j\neq k$

\item $[\rho_{ij},\lambda_{ik}]=1$ for all $i,j,k$

\item $[\rho_{ij}^{-1},\rho_{jk}^{-1}]=[\rho_{ij},\lambda_{jk}]=[\rho_{ij}^{-1},\rho_{jk}]^{-1}=[\rho_{ij},\lambda_{jk}^{-1}]^{-1}=\rho_{ik}^{-1}$

\item $[\lambda_{ij}^{-1},\lambda_{jk}^{-1}]=[\lambda_{ij},\rho_{jk}]=[\lambda_{ij}^{-1},\lambda_{jk}]^{-1}=[\lambda_{ij},\rho_{jk}^{-1}]^{-1}=\lambda_{ik}^{-1}$

\item $\lambda_{ij}\lambda_{ji}^{-1}\rho_{ij}=\rho_{ij}\rho_{ji}^{-1}\lambda_{ij}$

\item $(\rho_{ij}\rho_{ji}^{-1}\lambda_{ij})^4=1$.

\end{enumerate}

To get a presentation for $\aut$ we must add a generator $\tau$, corresponding to the automorphism $a_1\mapsto a_1^{-1}$,  and relations

\begin{enumerate}
\item[(7)]  $\tau^2=1$

\item[(8)]  $\tau\rho_{1j}\tau=\lambda_{1j}^{-1}, \tau\lambda_{1j}\tau=\rho_{1j}^{-1}$

\item[(9)] $\tau\rho_{i1}\tau=\rho_{i1}^{-1}, \tau\lambda_{i1}\tau=\lambda_{i1}^{-1}$

\item[(10)] $[\tau,\rho_{ij}]=[\tau,\lambda_{ij}]=1$ for $i,j\neq 1.$

\end{enumerate}
 
Finally, to get a presentation for $\Out$ we kill the inner automorphisms by adding the relation

(11)  $\displaystyle{ \prod_{i=2}^n\rho_{i1}\lambda_{i1}^{-1}=1}.$

We orient the petals of $R$ and label them with the generators $a_i$. If we fix a base vertex $\Oh$ of $L$, we may  think of $L$ as the 1-skeleton of the standard hypercubulation of $\R^n$ with vertices in $\Z^n$. The lift starting at $\Oh$ of the
edge labeled $a_i$ is identified with the standard $i$-th basis vector $\mathbf e_i$.

Any  automorphism $\phi$ of $F_n$ is realized on $R$ by a homotopy equivalence sending the  petal labeled $a_i$ to the (oriented) path which traces out the reduced word $\phi(a_i)$.  This has a {\it standard lift\,} $\widehat\phi$ to a $\Z^n$-equivariant homotopy equivalence of $L$, which sends  $\mathbf e_i$ to the lift starting at $\Oh$ of the path labeled by the reduced word $\phi(a_i)$.   (Since the homotopy equivalence is $\Z^n$-equivariant, it suffices to describe its effect on the edges $\mathbf e_i$.)  This in turn induces a lift $\widehat\phi_r$ to the quotient $L_r=L/\Z^r$ for each $r$, which is trivial in $\FHE(L_r)$ if and only if $\widehat \phi$ is fiberwise-homotopic to a deck transformation by an element of $r\Z^n$.

Lifting automorphisms to $L$  and $L_r$ by these standard lifts does not give a well-defined homomorphism on $\out $.  This is because the standard lift of the inner automorphism $\alpha_1=\prod_{i>1}\rho_{i1}\lambda^{-1}_{i1}$  sends $\mathbf e_i$  to a $\sqcup$-shaped path labeled  $a_1^{-1}a_ia_1$.  The extension to all of $L$ is freely homotopic to the deck transformation  ${\mathbf x}\mapsto {\mathbf x}-\mathbf e_1$ of $L$.  Since this deck transformation is  not freely homotopic to the identity (even mod $r$ for any $r>1$),  the assignment $\alpha_1\mapsto \widehat\alpha_1$ does not give well-defined map from $\out =\HE(R)$  to $\HE(L_r)$ (much less to $\FHE(L_r)$).

We rectify this situation by choosing lifts which are shifted from the standard lifts by appropriate translations of $L$.    Since $n-1$ is coprime to $r$, there are integers $s$ and $t$ with $s(n-1)+tr=1$.
We use the standard lift $P_{ij}=\widehat \rho_{ij}$ for $\rho_{ij}$, but for $\lambda_{ij}$ we choose the lift $\Lambda_{ij}$ which shifts the standard lift by $-s\mathbf e_j$, and for $\tau$ we choose the lift $T$ which shifts the standard lift by $s\mathbf e_1$.  Thus  on the vertices  $\mathbf v=(x_1,\ldots,x_n)$ of $L$, $P_{ij}$ acts as a shear parallel to the $\mathbf e_j$ direction, $\Lambda_{ij}$ is a shear composed with a shift, and $T$ is reflection across the hyperplane $x_1=s/2$.
In particular, each of our lifts  induces an affine map $\mathbf v\mapsto A\mathbf v + \mathbf b$, with $A\in {\rm{GL}}(n,\Z)$ and $\mathbf b\in\Z^n$.  Each edge   beginning at a vertex $\mathbf v$ in the direction $\mathbf e_i$ is sent to the path that begins at $A\mathbf v + \mathbf b$ and is labeled  $\phi(a_i)$.

We represent an affine map  $\mathbf v\mapsto A\mathbf v + \mathbf b$ by  the $(n+1)\times (n+1)$ matrix
$ \begin{pmatrix}A&\mathbf b\\\mathbf 0&1\end{pmatrix}$, acting on the vector $\begin{pmatrix} \mathbf v\\ 1\end{pmatrix}$.
Let $E_{pq}$ denote the $n\times n$ elementary matrix with one non-zero entry equal to $1$ in the $(p,q)$ position.  Thus the
action of $P_{ij}$ on the $0$-skeleton of $L$ is represented by the matrix with $A=I_n+E_{ji}$ and $\mathbf b=\mathbf 0$; for $\Lambda_{ij}$
we have the matrix with $A=I_n+E_{ji}$ and $\mathbf b = -s\mathbf e_j$; and
for $T$ the matrix with $A=I_n-2E_{11}$ and $\mathbf b=s\mathbf e_1$.

For example, for $n=2$ we have $s=1$ and $$P_{12} \sim\left(\begin{array}{c|c}
\begin{matrix}1 & 0  \\   1 & 1  \end{matrix} & \begin{matrix} 0 \\ 0\end{matrix}\\
\hline
\begin{matrix}0 & 0 \end{matrix} & 1
\end{array}\right)
\quad \Lambda_{12} \sim\left(\begin{array}{c|c}
\begin{matrix}1 & 0  \\   1 & 1  \end{matrix} & \begin{matrix} 0 \\ -1\end{matrix}\\
\hline
\begin{matrix}0 & 0 \end{matrix} & 1
\end{array}\right)
\quad\hbox{and}\quad T \sim \left(\begin{array}{c|c}
\begin{matrix}-1 & 0  \\   0 & 1  \end{matrix} & \begin{matrix} 1 \\ 0\end{matrix}\\
\hline
\begin{matrix}0 & 0 \end{matrix} & 1
\end{array}\right).
$$

\begin{remark}\label{r:edgesOK}
An important point to note is that since the relations (1) to (10) hold
in $\aut$ and not just $\out$, in order to verify that the above
assignments respect these relations we need only  verify that the
appropriate product of matrices is the identity: such a verification
tells us that the corresponding product of our chosen lifts acts trivially
on the vertices of $L$, and the action on edges (which is defined in
terms of the action on labels) is automatically satisfied. This remark
does not apply to relation (11), which requires special attention.
\end{remark}

\begin{proposition} For every integer $r$ coprime to $(n-1)$,
the  lifts    $P_{ij}$ of $\rho_{ij}$, $\Lambda_{ij}$ of $\lambda_{ij}$ and $T$ of $\tau$
define a splitting of the natural map $\FHE(L_r)\to \HE(R)=\out$.
\end{proposition}

\begin{proof}   We first claim that the maps
$\Lambda_{ij},P_{ij}$ and $T$ (and hence the maps
they induce on $L_r$) satisfy relations (1)
to (10).
In
each case, the verification is a straightforward calculation, which we  illustrate with several examples  using $j=2$ and $k=3$.
(In the light of remark \ref{r:edgesOK}, each verification simply
requires a matrix calculation.)

An example of a relation of type (4) is  $[\lambda_{12}^{-1},\lambda_{23}^{-1}]=\lambda_{13}^{-1}$.  

$$\left(\begin{array}{c|c}
\begin{array}{ccc}1 & 0 & 0 \\ -1 & 1 & 0 \\ 0 &0 & 1\end{array} & \begin{array}{c}0 \\ s \\ 0\end{array}\\
\hline
\begin{array}{ccc}0 & 0 & 0\end{array} & 1
\end{array}\right)
\left(\begin{array}{c|c}
\begin{array}{ccc}1 & 0 & 0 \\ 0 & 1 & 0 \\ 0 & -1 & 1\end{array} & \begin{array}{c}0 \\ 0 \\ s\end{array}\\
\hline
\begin{array}{ccc}0 & 0 & 0\end{array} & 1
\end{array}\right)
\left(\begin{array}{c|c}
\begin{array}{ccc}1 & 0 & 0 \\ 1 & 1 & 0 \\ 0 & 0 & 1\end{array} & \begin{array}{c}0 \\ -s \\ 0\end{array}\\
\hline
\begin{array}{ccc}0 & 0 & 0\end{array} & 1
\end{array}\right)
\left(\begin{array}{c|c}
\begin{array}{ccc}1 & 0 & 0 \\0 & 1 & 0 \\ 0 & 1 & 1\end{array} & \begin{array}{c}0 \\ 0 \\-s\end{array}\\
\hline
\begin{array}{ccc}0 & 0 & 0\end{array} & 1
\end{array}\right)
$$
$$
=
\left(\begin{array}{c|c}
\begin{matrix}1 & 0 & 0 \\ 0 & 1 & 0 \\ -1 & 0 & 1\end{matrix} & \begin{array}{c}0 \\ 0 \\ s\end{array}\\
\hline
\begin{matrix}0 & 0 & 0\end{matrix} & 1
\end{array}\right)
$$

To verify relation (6), we  first compute the action of $P_{12}P_{21}^{-1}\Lambda_{12}$ (we only need 2 indices),
$$
\left(\begin{array}{c|c}
\begin{matrix}1 & 0  \\   1 & 1  \end{matrix} & \begin{matrix} 0 \\ 0\end{matrix}\\
\hline
\begin{matrix}0 & 0 \end{matrix} & 1
\end{array}\right)
\left(\begin{array}{c|c}
\begin{matrix}1 & -1  \\   0 & 1  \end{matrix} & \begin{matrix} 0 \\ 0\end{matrix}\\
\hline
\begin{matrix}0 & 0 \end{matrix} & 1
\end{array}\right)
\left(\begin{array}{c|c}
\begin{matrix}1 & 0  \\   1 & 1  \end{matrix} & \begin{matrix} 0\\ -s\end{matrix}\\
\hline
\begin{matrix}0 & 0 \end{matrix} & 1
\end{array}\right)
=
\left(\begin{array}{c|c}
\begin{matrix}0 & -1  \\   1 & 0  \end{matrix} & \begin{matrix} s \\ 0\end{matrix}\\
\hline
\begin{matrix}0 & 0 \end{matrix} & 1
\end{array}\right),
$$
then check
$$
\left(\begin{array}{c|c}
\begin{matrix}0 & -1  \\   1 & 0  \end{matrix} & \begin{matrix} s \\ 0\end{matrix}\\
\hline
\begin{matrix}0 & 0 \end{matrix} & 1
\end{array}\right)
^4=
\left(\begin{array}{c|c}
\begin{matrix}1 & 0  \\   0 & 1  \end{matrix} & \begin{matrix} 0 \\ 0\end{matrix}\\
\hline
\begin{matrix}0 & 0 \end{matrix} & 1
\end{array}\right)
. $$

As an example of relation (8) we verify $TP_{12}T=\Lambda_{12}^{-1}$:
$$
\left(\begin{array}{c|c}
\begin{matrix}-1 & 0  \\   0 & 1  \end{matrix} & \begin{matrix} s \\ 0\end{matrix}\\
\hline
\begin{matrix}0 & 0 \end{matrix} & 1
\end{array}\right)
\left(\begin{array}{c|c}
\begin{matrix}1 & 0  \\   1 & 1  \end{matrix} & \begin{matrix} 0 \\ 0\end{matrix}\\
\hline
\begin{matrix}0 & 0 \end{matrix} & 1
\end{array}\right)
\left(\begin{array}{c|c}
\begin{matrix}-1 & 0  \\   0 & 1  \end{matrix} & \begin{matrix} s\\ 0\end{matrix}\\
\hline
\begin{matrix}0 & 0 \end{matrix} & 1
\end{array}\right)
=
\left(\begin{array}{c|c}
\begin{matrix}1 & 0  \\   -1& 1  \end{matrix} & \begin{matrix} 0 \\ s\end{matrix}\\
\hline
\begin{matrix}0 & 0 \end{matrix} & 1
\end{array}\right).
$$

Relation (11) is the only relation which requires some thought. For example, the matrix corresponding to the product $\prod_{i> 1}P_{i1}\Lambda_{i1}^{-1}$, which lifts conjugation by $a_1$, is
$$
\left(\begin{array}{c|c}
\begin{matrix}1 &0  \\   0 & I_{n-1}  \end{matrix} & \begin{matrix} (n-1)s\\  0\end{matrix}\\
\hline
\begin{matrix}0  &  0 \end{matrix} & 1
\end{array}\right).
$$
Thus for all $i>1$, the map on $L$ sends the edge starting at $\mathbf v$ in the direction $\mathbf e_i$   to the $\sqcup$-shaped path labeled $a_1^{-1}a_ia_1$ starting at $\mathbf v + s(n-1)\mathbf e_1=\mathbf v + \mathbf e_1 - tr\mathbf e_1$.  Dragging all vertices of $L$ one unit along the edge parallel to $\bf e_1$  gives a fiber-preserving homotopy of this map to the deck transformation  $\mathbf v\to \mathbf v -tr \mathbf e_1$.  This deck transformation induces the identity on $L_r$.  
\end{proof}

\begin{remark}
For $r=n=2$ the above construction
gives an embedding ${\rm{Out}}( F_2)\hookrightarrow {\rm{Out}}( F_5)$.  Here is an explicit description of the images of the $\rho_{ij},\lambda_{ij}$ and $\tau_i$ under this embedding, where $F_5=\langle a,b,c,d,e\rangle$.
$$\rho_{12}=\begin{cases} a\mapsto db \\
                           b\mapsto b\\
                           c\mapsto c\\
                           d\mapsto ea\\
                           e\mapsto e
                           \end{cases}
\quad
\rho_{21}=\begin{cases} a\mapsto a \\
                           b\mapsto cea\\
                           c\mapsto c\\
                           d\mapsto d\\
                           e\mapsto db
                           \end{cases}
\quad
\tau=\begin{cases} a\mapsto a^{-1} \\
                           b\mapsto e\\
                           c\mapsto c^{-1}\\
                           d\mapsto d^{-1}\\
                           e\mapsto b
                           \end{cases}
$$
$\lambda_{12}(x)=b\rho_{12}(x)b^{-1}$ and $\lambda_{21}(x)=a\rho_{12}(x)a^{-1}$.
\end{remark}

\section{Proof of Theorem~\ref{t:main}:  The non-existence of lifts}\label{s:no-split}

We begin by proving that for $n>2$ the map $\aut/N \to \out$ does not split when $N=F_n'$; this  is equivalent to the case $A\cong\Z^n$ in Theorem~\ref{t:main}.
To do this, we consider the
cyclic group $\Cyc < \out$ of order $(n-1)$ that corresponds to the group of rotations of the marked graph shown in Figure~\ref{Xn}.

\begin{proposition}\label{p:tf}   The inverse image of
$\Cyc$ in $\aut/F_n'$ is torsion-free, and therefore $\aut/F_n' \to \out$ does not split.
\end{proposition}

In this section  we present three  
proofs of this fact. The first is a geometric proof that we feel gives the most insight into
the non-splitting phenomenon; this is how we discovered Proposition \ref{p:tf}. The second proof draws attention
to a topological criterion illustrated by the first proof; like the first proof, it is executed using the lower sequence in Proposition \ref{p:sameProb}. The third proof is purely algebraic.  The first and third proofs also lead to a proof of the following proposition, which completes the proof of Theorem~\ref{t:split} (and therefore of Theorem~\ref{t:main}).

\begin{proposition}\label{p:not-coprime} Let $N = F_n'F_n^r$ and let $p_r$ denote the  natural map
$\aut/N\to \out$. Then the short exact
sequence $1\to F_n/N\to p_r^{-1}\Cyc\to \Cyc\to 1$ splits if and only if $r$ is coprime to $n-1$.
\end{proposition}

\subsection{A direct geometric proof }

At several points in the following argument we use the elementary fact that if a connected metric graph is a union of (at least two) embedded circuits, then an isometry that is homotopic to the identity is actually equal to the identity. 

Let $n\ge 3$ be an integer and let $X=X_n$  be the
graph that has $(n-1)$ vertices,
contains a simple loop of length $(n-1)$ and has a
loop of length $1$ at each of its vertices (see Figure~\ref{Xn}).
\begin{figure}
\begin{center}
\includegraphics[width=2in]{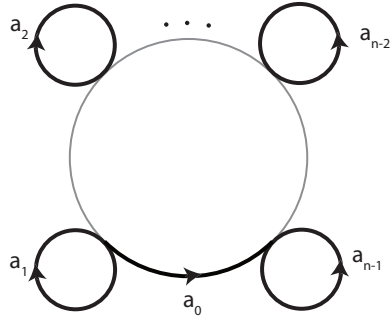}
\end{center}
\caption{The graph $X$}
\label{Xn}
\end{figure}

We fix a maximal tree in the graph, label the remaining
edge on the long circuit $a_0$, and label the loops of
length $1$ in cyclic order, proceeding around the long
cycle: $a_1,\dots, a_{n-1}$. This provides an identification
of $F_n$   with $\pi_1X$.

Consider the maximal abelian cover of $X$,
that is the graph $\widehat X=\widetilde X/F_n'$. The Galois group of this
covering is $F_n/F_n'\cong\Z^n$ and it is helpful to visualise
the following embedding of $\widehat X$ in $\R^n$ (see Figure~\ref{covering}).

\begin{figure}
\begin{center}
\includegraphics[width=4in]{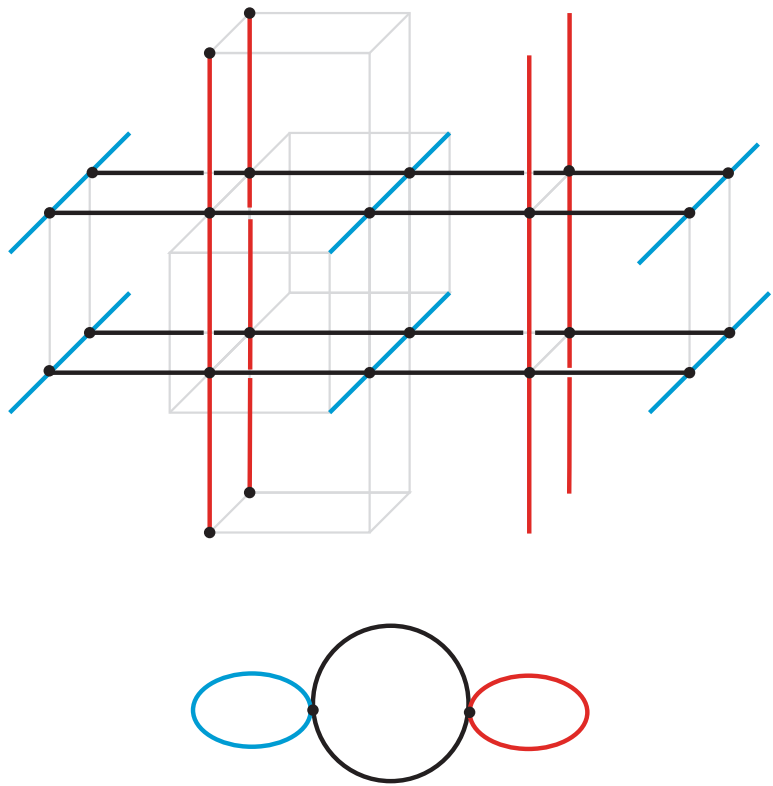}
\end{center}
\caption{Maximal abelian cover of $X=X_n$, $n=3$.}
\label{covering}
\end{figure}

Fix rectangular coordinates $x_0,\dots,x_{n-1}$ on $\R^n$
and define $\widehat X$ to be the union of the following $n$
families of lines: family $\mathcal{L}_0$ consists of
all lines parallel to the $x_0$-axis that have integer $x_i$-coordinates for all $i>0$, while $\mathcal{L}_i$ consists of
all lines parallel to the $x_i$ axis that have integer coordinates
for all $j\neq i$ with $j>0$ and which have $x_0$-coordinate
an integer that is congruent to $i \mod (n-1)$.

The action of the Galois group $F_n/F_n'\cong\Z^n$ is by translations in the
coordinate directions, with $a_i$ acting as translation by a distance
$1$ in the $x_i$ direction for $i=1,\dots,n-1$, and with $a_0$
acting as translation by a distance $(n-1)$ in the $x_0$ direction.

Now consider the isometry $\cyc$
of $X$ that rotates the long cycle
through a distance $1$, carrying the oriented loop labelled $a_i$
to that labelled $a_{i+1}$ for $i=1,\dots,n-2$
and taking $a_{n-1}$ to $a_1$.  This isometry has order $n-1$.  

A lift $\hat\cyc$ of $\cyc$ to $\widehat X$ is obtained as follows:
$$\hat\cyc (y_0,\dots,y_{n-1}) = (y_0+1,
y_{n-1}, y_1, y_2,\dots,y_{n-2}).$$
In other words, $\hat\cyc$ shifts by $1$ unit in the $x_0$-direction and permutes the
positive axes of the other generators cyclically.
In particular, $\hat\cyc^{n-1}$ is the deck transformation corresponding to  $[a_0]=(1,0,\ldots,0)\in \Z^n=F_n/F_n'$, so is not homotopic to the identity.  Any power of $\hat\cyc$ which is not a multiple of $n-1$ sends the axis for $a_1$ to a translate of the axis for $a_k$, for some $k\neq 1$, so is again not homotopic to the identity.  This shows that $\hat\cyc$ has infinite order in $\FHE(\widehat X)$.   

If we choose a different lift $\hat\cyc^\prime$ of $\cyc$, then it differs from $\hat\cyc$ by some deck transformation $(s,t_1,\ldots,t_{n-1})\in F_n/F_n'$.  Then $(\hat\cyc^\prime)^{n-1}$ is the
deck transformation $(1+s(n-1), t_1(n-1),\ldots, t_{n-1}(n-1)) $, which is non-trivial (hence not homotopic to the identity) for any $s$ if $n>2$. Thus $\hat\cyc^\prime$ has infinite order in $\FHE(\widehat X)$.   This proves Proposition \ref{p:tf}.

If we look mod $r$ (i.e. work modulo the action of $F_n/F_n'F_n^r= r\Z^n$), then
the last deck transformation considered above can  become trivial:  the equation $$1+s(n-1) \equiv 0 \mod r$$ has a solution if and only if $(r,n-1)=1$, so a lift $\hat\cyc_{(r)}$ of $\cyc$  to a fiber-homotopy equivalence of $\widehat X_{(r)}=\widehat X/r\Z^n$ can be chosen  so that $\hat\cyc_{(r)}^{n-1}$ is homotopic (in fact equal) to the identity  if and only if $(r,n-1)=1$.
This proves Proposition \ref{p:not-coprime}. \qed

\subsection{A topological obstruction to splitting}

The finite cyclic group generated by $\cyc$ acts freely on the graph $X_n$, and $X_n$ can be embedded into the torus $T^n$ in such a way that the  action   extends.   The kernel of the map induced on fundamental groups by this embedding is exactly the commutator subgroup $F_n'$.  Both $X_n$ and $T^n$ are aspherical spaces. In this section we show that the non-splitting of the short exact sequence of Proposition~\ref{p:tf} is an example of a more general phenomenon associated to this type of situation.

Let $G$ be a group acting freely by homeomorphisms on
a connected CW-complex $X$, and let $\widetilde X$ denote the universal cover.
Let $\widehat G\subset {\rm{Homeo}}(\widetilde X)$ be the
   subgroup
of $ {\rm{Homeo}}(\widetilde X)$ generated by all lifts of  elements of $G$.
(If the action of $G$ is properly discontinuous, then $\widehat G$
 is isomorphic to the fundamental group of $X/G$.)
There is an obvious short exact sequence
$$
1\to \pi_1X\to \widehat G\to G\to 1.
$$

More generally, if the action of $G$ leaves invariant a normal
subgroup $N\subset \pi_1X$ then we write $\widehat G_N$ for the
group of all lifts of the elements of $G$ to $\widetilde X/N$. There is
short exact sequence
$$
1\to \pi_1X/N\to \widehat G_N\to G\to 1,
$$
where $\pi_1X/N$ is the Galois group of the covering $\widetilde X/N\to X$.

\begin{lemma} The action of $\widehat G_N$ on $\widetilde X/N$ is free.
\end{lemma}

\begin{proof} If an element $\gamma\in\widehat G_N$ had a fixed point in $\widetilde X/N$
then its image in $G$ would fix a point of $X$. Since the action of $G$
is free, $\gamma$ would have to lie in the kernel of $\widehat G_N\to G$. But this kernel
is the group of deck transformations, which acts freely.
\end{proof}

\begin{lemma} \label{l:tf}
If $X$ is finite dimensional and
 aspherical then $\widehat G$ is torsion free.
\end{lemma}

\begin{proof} If $\widehat G$ had a non-trivial element of finite order, say $\gamma$,
then by the previous lemma we would have
a free action of the finite group $C=\<\gamma\>$ on the contractible
finite dimensional space $\widetilde X$, contradicting the fact that $C$
has cohomology in infinitely many dimensions.
\end{proof}

\begin{example} If $X$ is a graph and the action of $G$ is properly discontinuous
(e.g. by graph isometries) then $\widehat G$ is the fundamental
group of a graph and hence is free.
\end{example}

\begin{proposition}\label{p:aspherical}
Let $G$ be a group acting on finite-dimensional, connected CW-complexes   $X$ and $Y,$ and let  $f\colon X\hookrightarrow Y$ be an equivariant embedding.   Let $N$ be the kernel\footnote{This is well-defined as it is normal and
a change of basepoint isomorphism produces no ambiguity mod conjugacy.} of the induced map
 $\pi_1X\to \pi_1Y$ and consider the short exact sequence 
$$
1\to \pi_1X/N\to \widehat G_N\to G\to 1.
$$
 If $Y$ is aspherical and the action of $G$ on $Y$ is free, then 
$\widehat G_N$ is torsion-free.
\end{proposition}

\begin{proof} Let $\widehat G^Y$ be the group of all lifts to $\widetilde Y$
for the action of $G$ on $Y$.
The embedding $f:X\to Y$ lifts to an embedding
$\widetilde X/N\to \widetilde Y$ that induces an isomorphism
from $\widehat G_N$ to the subgroup of $\widehat G^Y$ that preserves the image
of $\widetilde X/N$. (This will be the whole of $\widehat G^Y$ if and only if
$f_*:\pi_1X\to\pi_1Y$ is surjective.)

Lemma \ref{l:tf} applied to $Y$ shows that $\widehat G^Y$ is torsion-free. 
\end{proof}

{\bf{Proof of Proposition \ref{p:tf}.}}  
We consider the graph $X_n$ shown in figure \ref{Xn} and
the  cyclic group $\Cyc$ of order $(n-1)$ that acts freely on
the graph, permuting the vertices in cyclic order. We embed $X_n$
in an $n$-dimensional torus $T$ by quotienting the embedding $\widehat X_n\to\R^n$
of the previous section by the action of $F_n/F_n'\cong\Z^n$ (see Figure~\ref{XinT}).

\begin{figure}
\begin{center}
\includegraphics[width=2in]{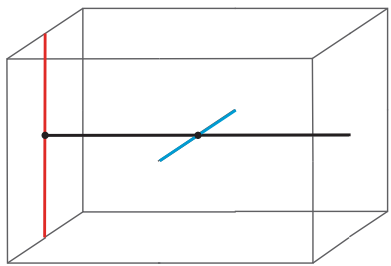}
\end{center}
\caption{$X_3$ embedded in a torus}
\label{XinT}
\end{figure}

We make the generator of $\Cyc$ act on $T$ by translation in the $x_0$ direction
through a distance $1$ followed by the rotation that leaves invariant the $x_0$
direction and permutes the coordinates $x_1,x_2,\dots,x_{n-1}$ cyclically.

This action in free and the embedding $X_n\to T$ is equivariant. Thus we are in the situation of Proposition~\ref{p:aspherical}, and Proposition \ref{p:tf} is proved.

\subsection{A proof using presentations}

We are interested in the short exact sequence
$$
1\to F_n/F_n' \to \aut/F_n'\overset{\pi}\to \out \to 1 .
$$
Let $\Cyc\subset\out$ be the subgroup generated
by the class of the automorphism
$$\cyc:(a_0,a_1,\dots,a_{n-2},a_{n-1})
\mapsto (a_0,a_2,\dots,a_{n-1}, a_0a_1a_0^{-1}).
$$
Note that $\Cyc$ is cyclic of order $(n-1)$ but that
$\cyc$ has infinite order in $\aut$, since it is a
root of the inner automorphism by $a_0$.

Let $\hCyc = \pi^{-1}\Cyc\subset\aut/F_n'$, so the above short exact sequence restricts to:
$$
1\to F_n/F_n' \to \hCyc\overset{\pi}\to \Cyc \to 1 .
$$
We produce a presentation for   $\hCyc$  using a standard
 procedure  for constructing presentations of group extensions; this
is explained, e.g.,  in \cite{Johnson}, Theorem 1, p. 139.   
We fix a
basis $\{a_0,\dots,a_{n-1}\}$ for the free group
$F_n$ and write $\alpha_i$ for the image in  $\aut/F_n'$
of the inner automorphism $w\mapsto a_iwa_i^{-1}$.  Then $F_n/F_n'$ is generated by the $\alpha_i$ subject to the relations $[\alpha_i,\alpha_j]=1,$ and $\Cyc$ is generated by the image of $\theta$ subject to the relation that this image has order $n-1$.    The automorphisms $\alpha_i$ and $\cyc$ satisfy the following relations:
\begin{enumerate}
\item $\cyc\a_{0}\cyc^{-1} = \a_0$
\item $\cyc\a_i\cyc^{-1}=\a_{i+1}$ for $i=1,\ldots,n-2$
\item $\cyc\a_{n-1}\cyc^{-1} = \a_0^{-1}\a_1\a_0$
\item $\cyc^{n-1}=\a_0$
\end{enumerate}
and the theorem cited above assures us that (introducing a
generator $x$ to represent $\cyc$) these
relations suffice to present $\hCyc$:
$$
 \hCyc \cong \<\a_0,\dots,\a_{n-1},x \mid [\a_i,\a_j]=1 {\text{ for $i,j=0,\dots,n-1$}},$$
$$x\a_{0}x^{-1} = \a_0,\,\, x\a_i x^{-1}=\a_{i+1} {\text { for $i=1,\ldots,n-2$}}, $$
$$
x\a_{n-1}x^{-1} = \a_0\a_1\a_0^{-1},\,\, x^{n-1}=\a_0\rangle
$$

\begin{proposition}\label{p:pres}
$\hCyc \cong \Z^{n-1}\rtimes_\psi\Z$ where
$\psi$ is the automorphism that permutes a free basis
$\{\a_1,\dots,\a_{n-1}\}$ cyclically.  In particular, $\hCyc$ is torsion-free.
\end{proposition}

\begin{proof}  We use  Tietze moves to simplify our presentation of $\hCyc$. First we use $[\a_0,\a_1]=1$ to replace
$x \a_{n-1} x^{-1} = \a_0\a_1 \a_{0}^{-1}$ by
$x\a_{n-1} x^{-1}  = \a_1$. Next we use the last relation
to remove the superfluous generator $\a_0$, replacing
it by $x^{n-1}$ in the other relations where it appears.
But in fact, all of the relations where $\a_0$ appeared
become redundant when we substitute $x^{n-1}$: this
is obvious for  $x\a_{0}x^{-1} = \a_0$, and in the remaining cases
one can deduce $[x^{n-1},\a_i]=1$ by combining the
relations $x \a_i x ^{-1}= \a_{i+1}$ and
$x\a_{n-1} x^{-1} = \a_1$.

At the end of these moves we are left with the presentation
$$ \hCyc \cong \<\a_1,\dots,\a_{n-1},x
\mid [\a_i,\a_j] {\text{ for $i,j=1,\dots,n-1$}}, \
$$
$$
x \a_i x^{-1} = \a_{i+1} {\text{ for $i=1,\dots,n-2$}}, \ x \a_{n-1} x ^{-1}= \a_1\>,
$$
which is the natural presentation of 
$\Z^{n-1}\rtimes_\psi\Z$.
\end{proof}

\begin{corollary} $\hCyc$ is the fundamental group of
a closed, flat $n$-manifold that fibres over the circle with holonomy of order
$(n-1)$.
\end{corollary}

\noindent{\bf Proof of Propositions \ref{p:tf} and \ref{p:not-coprime}}

From our original presentation of $\hCyc$
we readily
deduce the following presentation
for the preimage $\Cyc(n,r)$ of $\Cyc$ in $\aut/F_n'F_n^r$
$$\<\a_0,\a_1,\dots,\a_{n-1},x
\mid \a_i^r=1=[\a_i,\a_j]  {\text{ for $i=0,\dots,n-1$}},$$
$$
x \a_i x^{-1} = \a_{i+1} {\text{ for $i=1,\dots,n-2$}},$$
$$
x \a_{n-1} x^{-1} = \a_1,\ \a_0=x^{n-1}  \>,
$$
and making Tietze moves as above
we see that $\widehat \Cyc(n,r)$ is a semidirect product
$\widehat \Cyc(n,r) \cong (\Z/r)^{n-1}\rtimes_\psi\Z/{r(n-1)}$; in particular
$x$ has order $r(n-1)$.

Let $N=F_n/F_n'F_n^r\cong (\Z/r)^n$ denote the subgroup generated by the $\a_i$.
We are interested in when we can split
$$
1\to N\to \widehat \Cyc(n,r)\to \Cyc\to 1.
$$
If $r$ is coprime to $n-1$, then there exists an integer $t$
such that $tr \equiv 1 \mod (n-1)$, so $[\theta]^{tr} = [\theta]$ in $\Cyc$
and we can split the above sequence by sending the generator $[\theta]\in\Cyc$
to $x^{tr}$, noting that  
$$
(x^{tr})^{n-1}= (x^{r(n-1)})^t=1.
$$

It remains to prove that if $r$ is not coprime to $n-1$ then there is
no splitting. To establish this, we consider an arbitrary element in
the preimage of $[\theta]$ and examine whether it can have order
$n-1$. Such an element has the form
$vx$, where $v=\a_0^{m_0}\a_1^{m_1}\dots \a_{n-1}^{m_{n-1}}$.
From our presentation of $\Theta(n,r)$ we see that
$$
(vx)^{n-1} = v.(x vx^{-1}).(x^2vx^{-2}).(x^3vx^{-3}).\dots.(x^{n-2}vx^{2-n}).
x^{n-1}$$
can be simplified to
$$
(vx)^{n-1} = v.\psi(v).\psi^2(v).\psi^3(v).\dots.\psi^{n-2}(v).
a_0.
$$
And since
$$\a_i.\psi(\a_i).\psi^2(\a_i).\psi^3(\a_i).\dots.\psi^{n-2}(\a_i) = \mu:=
\a_1\a_2\dots\a_{n-1}$$
for $i=1,\dots,n-1$, while $x\a_0x^{-1}=\a_0$, we have
$$
(vx)^{n-1}  = \a_0^{m_0(n-1)+1} \mu^{m_1+\dots+m_{n-1}}.
$$
In order for this to equal the identity in $\Cyc(n,r)$ the exponent
of $\a_0$ has to be zero mod $r$. But this is impossible, because
$n-1$ is not coprime to $r$ and hence there is no integer $m_0$
such that $m_0(n-1)+1\equiv 0 \mod r$. \qed

\section{Theorem B: A cohomological remark}\label{s:Mdual}

Let $M=H_1(F_n)$ be the standard left-module for the left action of $\out$.  In the previous section we exhibited
an extension $1\to M\to \aut /F_n'\to \out \to 1$ which does not split
and therefore determines a non-trivial cohomology class in $H^2(\out ;M)$; this proves the first statement of Theorem~\ref{t:cohomology}.  For the second statement,  we consider the dual $M^*= H^1(F_n)$ of the 
standard module.

\begin{proposition} $H^2(\out ,M^*)=0$ for $n\geq 8$.
\end{proposition}

\begin{proof}
To compute $H^2(\out ,M^*)$, we use the Hochschild-Lyndon-Serre  spectral sequence in cohomology for the short exact sequence
$$1\to F_n\to \aut \to \out \to 1,$$
with trivial $\Z$ coefficients.
This has $E_2^{p,q}=H^p(\out ;H^q(F_n)) \Rightarrow H^{p+q}(\aut )$.  Since $H^q(F_n)=0$ for $q>1$, the  $E_2$-term has exactly two non-zero rows, for $q=0$ and $q=1$:
\medskip

\,\,\begin{xy}
(0,0) ; (145,0) **\dir{-};
(0,35); (0,0) **\dir{-};
\xymatrix@C=1pc
{
&0 \ar[drr]^{d_2} & 0 \ar[drr]^{d_2} &0  &0\\
&H^0(\out ;M^*) \ar[drr]^{d_2} & H^1(\out ;M^*) \ar[drr]^{d_2} &H^2(\out ;M^*) & H^3(\out ;M^*)\\
&H^0(\out ;\Z)  & H^1(\out ;\Z)  &H^2(\out ;\Z)  & H^3(\out ;\Z)
}
\end{xy}

\medskip
The $E_2^{p,0}$ terms are $H^p(\out ;\Z)$ with trivial $\Z$-coefficients, and the $E_2^{p,1}$ terms are $H^p(\out ;M^*)$.   Now
$$E_\infty^{p,0}=E_3^{p,0}=H^p(\out ;\Z)/im(d_2).$$
Since the spectral sequence converges to the cohomology of $\aut $, we have a two-stage filtration
$$0\subset E_\infty^{p,0}\subset H^p(\aut ;\Z) \quad \hbox{with}\quad  E_\infty^{p-1,1}=H^{p-1}(\aut ;\Z)/E_\infty^{p,0}.$$
The map on cohomology  induced by  $\aut  \to \out $ factors through the edge homomorphism $e\colon E^{p,0}_\infty\to H^p(\aut ;\Z)$:
\begin{equation*}
\xymatrix@C=1pc@R=1pc
{H^p(\out ;\Z)\ar[rr]\ar@{->>}[dr] && H^p(\aut ;\Z)   \\
&H^p(\out ;\Z)/im(d_2)\ar@{^(->}[ur]^e  }
\end{equation*}

But the top arrow is an isomorphism for $n>\!>p$ (\cite{HV,HVW}), so in this range all of these maps are isomorphisms  and $d_2=0$.  Applying this with $p=2, 3$ and $4$ we see that  $E_3^{2,1}=E_\infty^{2,1}=H^2(\out ;M^*)$ must be zero.

The exact stable range for $H^p(\out )$ is still unknown.  A lower bound,
from \cite{HV}, is $n\geq 2p+4$, which gives $n\geq 12$ when $p\leq 4$.
\end{proof}

The form of the cohomology argument above may be abstracted as follows.

\begin{lemma} \label{l:zero}
Let $1\to F\to \G\overset{\pi}\to Q\to 1$ be a short exact sequence
with $F$ a free group, and let $M\cong H^1(F,\Z)$ be the associated
$\Z Q$-module. If $\pi$ induces agisomorphism $H^{p-1}(Q;\Z)
\to H^{p-1}(\G;\Z)$ and an injection $H^{p}(Q;\Z)
\to H^{p}(\G;\Z)$, then $H^{p-2}(Q;M)=0$.
\end{lemma}

\section{Theorem~\ref{t:nomaps}: Classification of  graphs realizing finite subgroups}  

In the course of this section and the next
 we shall prove that if $n$ is even and $n<m\le 2n$, or $n$ is odd and $n<m\le 2n-2$,
then every homomorphism $\out\to\outm$ has  image of order at most two.
We do this by examining the possible images in $\outm$ of the finite subgroups of $\out$. We show that the possible embeddings of the largest
finite subgroups of $\out$  are so constrained that none can be extended to a homomorphism defined on the whole of $\out$.
Arguing in this manner, we deduce that no homomorphism from $\out$ to $\outm$ can restrict to an injection on the largest finite subgroup $W_n\subset\out$. This enables us to
apply results from our previous work \cite{BV3}, in which we described the homomorphic images of $\out$ into which $W_n$ does not inject.

\subsection{Admissible graphs}

\begin{definition}
A graph is {\em{admissible}} if it is finite, connected, has no vertices of valence 1 or 2, and
has no non-trivial forests that are invariant under the full automorphism group of the graph.
An admissible graph on which a group $G$ acts is said to be {\em $G$-minimal} if the action is faithful
and there are no forests which are invariant under the $G$-action; thus every admissible graph is minimal for its full automorphism group.

\end{definition}

Note that an admissible graph can have no separating edges, so our
notion is more restrictive than the notion of admissible used in \cite{CV}.

The following theorem explains our interest in admissible graphs.

\begin{theorem}[\cite{Culler,Khramtsov}]\label{t:real} Every finite subgroup of $\Out$ can be realized as a subgroup of the automorphism group of an admissible graph with fundamental group $F_n$.
\end{theorem}

An easy exercise using Euler characteristic yields:

\begin{lemma} An admissible graph of genus $m$ has at most $2m-2$ vertices and $3m-3$ edges.
\end{lemma}

The {\it genus} of a graph  $X$ is the rank of $H_1(X)$.  It can be computed as $e-v+c$, where $e$ is the number of edges of $X$, $v$ is the number of vertices and $c$ is the number of components.

\begin{lemma}\label{genus} A proper subgraph of an admissible graph has strictly smaller genus.   
\end{lemma}

\subsection{Classification of admissible $A_n$-graphs}

We are interested in finite subgroups  of $\outm$ that contain alternating groups, so we begin by classifying  admissible graphs of genus $m\leq 2n$ which {\em realize} the alternating group $A_n$, i.e. graphs which admit a faithful action of $A_n$ by isometries.

Two graphs which admit obvious $A_n$-actions are the $n$-cage $C_n$, which has two vertices and $n$ edges joining them, and the $n$-rose $R_n$ which has one vertex and $n$ loops   (see Figure~\ref{RoseCage}).  These will appear frequently in our discussion of $A_n$-graphs.  
\begin{figure}
\begin{center}
\includegraphics[width=2.5in]{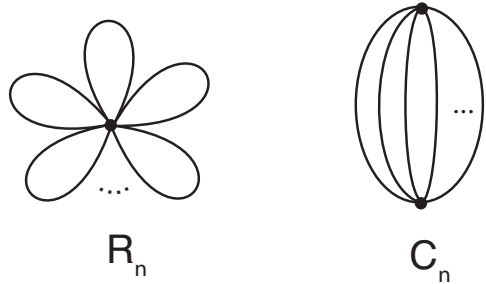}
\end{center}
\caption{An $n$-rose $R_n$ and and $n$-cage $C_n$}
\label{RoseCage}
\end{figure}

If $X$ is a graph with an $A_n$-action, we denote the orbit of a vertex $v$ by $[v]$.  In the next lemma we consider orbits of cardinality $n$.  We use the fact that the action of $A_n$ on a set of size $n$ is either trivial or standard, provided $n\neq 4$.  (For $n=4$, however,  $A_4$ has the Klein 4-group as a normal subgroup, with quotient $\Z/3$, which acts on four points by fixing one of them.)

\begin{lemma}\label{sizen} Suppose $n\geq 5$, and let $X$ be a  graph of genus $m < (n-1)(n-2)/2$ which realizes $A_n$.
If all vertex-orbits $[v]$ have size $n$, then $X$ is the disjoint union of $n$ subgraphs that are permuted by the action of $A_n$ in the standard way.
\end{lemma}

\begin{proof}
Since $n\geq 5$, the action of $A_n$ on each orbit $[v]$ is the standard permutation action. In particular, the stabilizer of each vertex $v$ is isomorphic to $A_{n-1}$, and acts transitively on the other vertices in $[v]$.  
Moreover these point stabilizers account for all of the subgroups of $A_n$ that are isomorphic to $A_{n-1}$.

 Fix a vertex $v_0$.  In each vertex-orbit $[w]$ there is a unique vertex $w_0\in [w]$ whose stabilizer is the same as that of $v_0$.  Let $X_0$ be the subgraph spanned by all of the $w_0$, including $v_0$. We claim that $X$ is the disjoint union of copies of $X_0$ permuted by the action of $A_n$.

 If a vertex $w_0\in X_0$ is connected to a vertex $u$ outside of $X_0$ by an edge, then the orbit of $u$ under the stabilizer of $w_0$ has $n-1$ elements, so $w_0$ has valence at least $n-1$;  similarly $u$ has valence at least $n-1$.  Let $X_1$ be the subgraph spanned by $[w]$ and $[u]$.  If $[u]=[w]$, then $X_1$ contains the complete graph  on $n$ vertices; but this graph has genus  $(n-1)(n-2)/2>m$, so this is impossible.  If $[u]\neq [w]$,  the genus of $X_{1}$ is at least $n(n-1)-2n+1$, but again this genus is strictly bigger than $m$ so this is impossible.  

It follows that $X$ is the disjoint union of $n$ copies of $X_0$, one for each $v\in [v_0]$, and that these are permuted by the action of $A_n$ in the standard way.
 
\end{proof}

If $X$ realizes $A_n$, then $A_n$ acts on the set of vertices and on the set of edges of $X$.  Our analysis of $A_n$-graphs depends on the following result of M. Liebeck.

\begin{proposition}[\cite{Liebeck}, Prop. 1.1]\label{p:Liebeck}   If $n>8$ then the orbits of the action of $A_n$ on a finite set $S$ have size $1$, $n$, $n\choose 2$ or larger.  If $n=7$ or $8$ there may also be an orbit of size $15$.  If $n=6$  there may be an orbit of size $10$.
\end{proposition}

In the following proposition, the names of  graphs refer to Figure~\ref{Anfigure}. In each case, 
the automorphism group of the graph contains
a unique copy of $A_n$, up to conjugacy. 
(This can be seen by an elementary argument starting with the observation that in each case there is only a single possible non-trivial vertex orbit.)  Thus, for the most part, we need not specify how $A_n$ is acting each time such a graph appears.

 We use the following standard notation: if $X_1$ and $X_2$ are graphs, each with a distinguished vertex, then we write
$X_1\vee X_2$ for the graph obtained from the disjoint union $X_1\sqcup X_2$ by identifying these vertices;  if each of $X_1$ and $X_2$ is equipped with an action by a group $G$, we refer to the
induced action on $X_1\sqcup X_2$ and $X_1\vee X_2$ as the diagonal action.
\begin{figure}
\begin{center}
\includegraphics[width=5in]{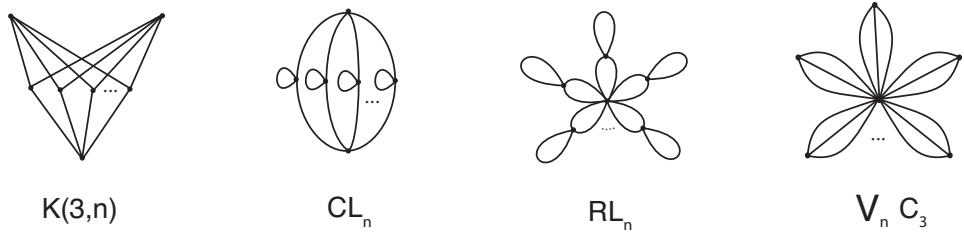}
\end{center}
\caption{Admissible graphs of rank $\leq 2n$ realizing $A_n$}
\label{Anfigure}
\end{figure}

\begin{proposition}[Classification of admissible $A_n$-graphs]\label{p:An} Suppose $n>8$, and let $X$ be an admissible graph of genus $m\leq 2n$ which realizes $A=A_n$.
 Let $X_A$ be the subgraph of $X$ spanned by edges with non-trivial $A$-orbits.  
\begin{enumerate}
\item If $m<n-1$ there are no admissible graphs realizing $A$.
\item If $n-1 \leq m<2n-2$ then $X_A=R_n$ or $C_n$.
\item If $m=2n-2$ then  $X_A$ is $R_n$, $C_n$,  $C_n\vee C_n$ or  $K(3,n)$.  
\item If $m=2n-1$ then $X_A$ is one of the above or $CL_n$, or  is   $C_{2n}$, $C_n\vee R_n$ or $C_n\sqcup C_n$ with diagonal action.  
\item If $m=2n$ then $X_A$  is one of the above,  $R_{2n}=R_{n}\vee R_n$ with diagonal action, $R_n \sqcup C_n$, $RL_n$ or $\bigvee_nC_3$.
\end{enumerate}
$X$ is obtained from $X_A$ by adding additional edges and vertices, fixed by $A$, in an arbitrary manner subject to the requirement that $X$
must be connected and must not contain a  non-trivial forest that is invariant under the action of ${\rm{Aut}}(X)$.
\end{proposition}

\begin{proof}
Since $X$ is admissible of genus $m\leq 2n$, it has at most $2m-2\leq 4n-2$ vertices, and since $n>8$ all vertex orbits have size $1$ or $n$. Therefore there are at most three non-trivial vertex orbits, and
we divide the classification into cases according to the number of these.

\medskip
{\bf Case 1: All vertices of $X$ are fixed.} In this case the subgraph $X_A$ is a union of cages and roses.  Since $n>8$ and the genus of $X$ is at most $2n<{n\choose 2}$, these cages and roses  must have exactly $n$ edges (Proposition~\ref{p:Liebeck}).  The genus of $X_A$ gives a lower bound on the genus of $X$, so  the genus of $X$ is at least $n-1$. If $n-1\leq m<2n-2$, the graph $X$   contains exactly one cage or one rose.  

If $m=2n-2$ then  the only new possibility is $X_A=C_n\vee C_n=X$, with diagonal action of $A_n$.  

If $m=2n-1$ we must consider the new possibilities $X_A=R_n\vee C_n$, $X_A=C_{2n}$ and $X_A=C_n\sqcup C_n$. If $X_A=  C_n\sqcup C_n$ then $X$ is obtained from $X_A$
by adding two extra (fixed) edges, which must both have the same endpoints, since otherwise the full group of isometries of $X$ would have an invariant forest and hence X would not be admissible.   If $X_A=C_n\vee C_n$ then
$X$ has one fixed edge and this cannot join the vertices of the $C_n$ which lie opposite the wedge vertex, for the same reason.  

If $m=2n$ the only new possibility for $X_A$ is $R_n\vee R_n=X$, where the action of $A_n$ is diagonal.   

\medskip
If there are non-trivial vertex orbits, consider the (invariant) subgraph of $X_A$ obtained by deleting all fixed vertices (and adjacent edges) of $X$.  By Lemma~\ref{sizen}, this subgraph is a disjoint union of  subgraphs $X_1,\ldots, X_n$   which are permuted by $A_n$.  Since $m\leq 2n$, each of these subgraphs  can have genus at most $1$.

\medskip
{\bf Case 2:  $X$ has one non-trivial vertex orbit.} In this case  $X_1$ has only a single vertex $v$, possibly with one loop attached.

If $v$ has a loop attached, there must be at least two other edges of $X$ adjacent to $v$ (since $X$ has no separating edges), and each of these edges has its other end at a fixed vertex.  If these vertices are  the same, then $X=X_A$ is  the ``rose with loops"  $RL_n$  which has genus $2n$.  If they are different, then $X$ contains the ``cage with loops" $CL_n$, so has genus  at least $2n-1$.  

If there is no loop at $v$, there must be at least three edges $e_1,e_2$ and $e_3$ of $X$ adjacent to $v$, terminating at fixed vertices $u_1, u_2$ and $u_3$.
If  these vertices all coincide,   then $e_1,e_2$ and $e_3$ form a 3-cage, whose $A_n$-orbit is a copy of $\bigvee_n C_3$ in $X$; since this has genus $2n$ it is in fact all of $X$.  If $u_1, u_2$ and $u_3$ are distinct, then the $A_n$-orbit of $e_1, e_2$ and $e_3$ forms a copy of $K(3,n)$, which has genus $2n-2$ and is all of $X_A$ since there is no room for another non-trivial edge orbit.
The case $u_1=u_2\neq u_3$ cannot occur, since the orbit of $u_3$ would be a forest invariant under the full isometry group of the graph.

\medskip
{\bf Case 3: $X$ has two non-trivial vertex orbits.}  In this case $X_1$ has two vertices $v$ and $w$, and we claim this can never give an admissible graph of rank $\leq 2n$.   If $X_{1}$ is a 2-cage $C_2$,  then  there must be another edge starting at $v$ and another edge at $w$, since $X$ is admissible.  These edges may terminate at the same or at different fixed points.  In either case, their orbits  form a forest invariant under the full isometry group of $X$.  Other possibilities for $X_{1}$  are eliminated by using the  fact that $X$ is admissible to count the minimal number of  orbits of edges terminating in $X_1$, then estimating the genus of the subgraph spanned by these edge-orbits;  in all cases, this genus is  bigger than $2n$.  For example, if $X_{1}$ is a single edge, there must be at least $4$ more edges adjacent to $X_{1}$,  and all must be in different edge-orbits since there are no orbits of size $2n$.  The subgraph spanned by the orbit of $X_1$ and these additional edge-orbits has $5n$ edges and at most $2n+4$ vertices, so its genus is at least $3n-3>2n$.  

\medskip
{\bf Case 4: $X$ has three non-trivial vertex orbits}. This case also cannot occur.  Let  $u, v$ and $w$ be the vertices of $X_1$.     In all cases, the fact that $X$ is admissible allows us to find a subgraph of  genus greater than $2n$.  For example, if $X_{1}$ is a triangle, there are at least  $3$ additional edges terminating in $X_1$.  The subgraph spanned by the orbits of $X_{1}$ and these additional edges has $6n$ edges and at most $3n+3$ vertices, so has genus at least $3n-2>2n$.   

\end{proof}

\def\o{\omega}
\def\O{\Omega}

\subsection{Classification of minimal admissible $W_n$-graphs }

Let $W_n\cong(\Z/2)^n\rtimes S_n$ be the full group of automorphisms of $R_n$.  If we identify $R_n$ with the standard rose with petals labelled by the generators of $F_n$, the subgroup $S_n$ is generated by permutations of the generators and the subgroup $(\Z/2)^n$ is generated by the automorphisms $\varepsilon_i$, where $\e_i$ inverts the $i$-th generator.   In this section  we classify all minimal admissible $W_n$-graphs $X$ of genus $m\leq 2n$.

\begin{lemma}\label{Snaction} Suppose $S_n$ acts on a finite set $\Omega$.  Then $S_n$  permutes the $A_n$-orbits in $\Omega$, and the action on this set of orbits factors through the determinant map $S_n\to S_n/A_n\cong \Z/2$.
\end{lemma}
\begin{proof} For all $\sigma\in S_n$ and $\o\in \O$ we have $\sigma(A_n \o)=A_n(\sigma \o)$ since $A_n$ is normal. 
\end{proof}

\begin{lemma}\label{Wnaction} Suppose $W_n$ acts on a finite set $\O$, and $A_n$ has a single non-trivial orbit $\O_A$ of size $n$.  Then $\O_A$ is invariant under the full group $W_n$.  Each $\varepsilon_i$ acts as the identity on $\O_A$,  all the $\varepsilon_i$ act by the same involution\footnote{for brevity, we use the term ``involution" to mean a symmetry that either has order 2 or is the identity}  on the fixed set $\O^A$ of the $A_n$-action, and every transposition in $S_n$ acts by the same involution on $\O^A$.   
\end{lemma}

\begin{proof}   By Lemma~\ref{Snaction}  the action of $S_n< W_n$ preserves $\O^A$ and $\O_A$, and all permutations of determinant $-1$ act by the same involution of $\O^A$.  Thus it only remains to check the action of  the $\varepsilon_i$.  

Let $\o_1,\ldots, \o_n$ be the elements of $\O_A$, with the standard $A_n$ action on the subscripts.  The centralizer of $\varepsilon_1$ contains a copy of $A_{n-1}$, so $\varepsilon_1$ acts on the fixed point set of this $A_{n-1}$, which is  $\O^A\cup \omega_1$.    Assume that $\varepsilon_1$ sends   $\o_1$  to $t\in \O^A$; we will show that this leads to   a contradiction.  Set $\sigma=(12)(ij)$ for some $i\neq j>2$. Then $\sigma\in A_n$ and $\varepsilon_2=\sigma\varepsilon_1\sigma$.  Applying this to $\o_2$ shows that $\varepsilon_2(\o_2)=t$.  Thus $\varepsilon_1\varepsilon_2(\o_2)=\varepsilon_1(t)=\o_1$.  Since $\varepsilon_1$ and $\varepsilon_2$ commute,   this gives $\varepsilon_2\varepsilon_1(\o_2)=\o_1$, i.e. $\varepsilon_1(\o_2)=\varepsilon_2(\o_1)$, which implies that  $\varepsilon_1(\o_2)\neq \o_1, \o_2$ or $t$.
Thus $\varepsilon_1(\o_2)=\o_i$ for some $i>2$.  Now $\varepsilon_1\varepsilon_2=\varepsilon_1\sigma\varepsilon_1\sigma$ and $\varepsilon_2\varepsilon_1=\sigma\varepsilon_1\sigma\varepsilon_1$; applying the first expression to $\o_2$ gives $\o_1$, but the second expression sends $\o_2$ to $\sigma\varepsilon_1(\o_i)\neq \o_1$, giving   the desired contradiction.  We conclude that $\varepsilon_1(\o_1)=\o_1$ and $\varepsilon_1(\o_i)\in \{\o_2,\ldots, \o_n\}$ for all $i>1$, i.e. $\varepsilon_1$ preserves $\O_A$ and $\O^A$.  

In fact, we must have $\varepsilon_1(\o_i)=\o_i$ for all $i$. To see this, suppose, e.g., that $\varepsilon_1(\o_2)=\o_3$.  Then
 \begin{align*}
\o_3=\varepsilon_1(\o_2)=\varepsilon_1\varepsilon_2(\o_2)&=\varepsilon_2\varepsilon_1(\o_2) \\ =\varepsilon_2(\o_3)
&=(12)\varepsilon_1(12)(\o_3) = (12)\varepsilon_1(\o_3)=(12) (\o_2)=\o_1,
 \end{align*}
 giving a contradiction.  

Since all $\varepsilon_i$ are conjugate by elements of $A_n$, they all act in the same way on $\O^A$.   
\end{proof}

Now let $X$ be a minimal admissible $W_n$ graph.  As in the previous section, we denote by $X^A$ the subgraph fixed by the $A_n$-action  and by $X_A$ the subgraph of $X$ spanned by edges in non-trivial $A_n$-orbits.  
Note that $X=X^A\cup X_A$.

\begin{notation} Let $\Delta = \e_1\dots\e_n\in W_n$ and let  $\alpha : W_n\to W_n$ be the homomorphism that  is the identity on $S_n<W_n$ and sends each $\e_i$ to $\e_i\Delta$.
\end{notation}
Note that $\alpha$ is an automorphism if $n$ is even but has kernel $\<\Delta\>$ if $n$ is odd.
\smallskip

In light of Theorem \ref{t:real}, the following proposition
provides a complete description, up to conjugacy, of the
subgroups of $\outm$ isomorphic to $W_n$ with $n>8$ and $m\le 2n$.

 \begin{proposition}\label{p:Wn} Suppose $n>8$ and let
$X$ be a $W_n$-minimal, admissible graph of genus $m\leq 2n$.    Then  $X_A$ is invariant under the whole group $W_n$, and
 all of the  $\varepsilon_i$ have the same
restriction to $X^A$. The possibilities for $X_A$ are:
\begin{enumerate}
\item If $m\leq 2n-2$  then $X_A=R_n$ and the action of $W_n$ on $R_n$ is either the standard one or else the standard one twisted by $\alpha:W_n\to W_n$.
\item  If $m=2n-1$, the only additional possibilities for  $X_A$ are $C_{2n}$, $R_n\vee C_n$,  and $CL_n$. In all cases, $X=X_A$. In the action on $C_{2n}$, the edges are grouped in
pairs $\{e_i,e_i'\}$ so that the action of $\sigma\in S_n$ sends $e_i$ to $e_{\sigma(i)}$ and $e_i'$ to $e_{\sigma(i)}'$, and either the action of $\e_i$ is standard (i.e. it exchanges $e_i$ and $e_i'$ only)
or else it is the standard action twisted by  $\alpha:W_n\to W_n$. In addition, the $\e_i$ and the transpositions in $S_n$ may exchange the vertices of
$C_{2n}$. The action of $W_n$ on $R_n\subset R_n\vee C_n$ is as in (1) and the $\e_i$ act trivially on $C_n\subset R_n\vee C_n$.
In a standard action of $W_n$ on $CL_n$, each $\e_i$ flips the $i$th loop and leaves all others other fixed, and $W_n$ interchanges the vertices of $C_{2n}$
via a non-trivial homomorphism $W_n\to \Z/2$; any action of $W_n$
on $CL_n$ is either standard or
else a standard one twisted by $\alpha:W_n\to W_n$.
\item If $m=2n,$ the only additional possibilities for $X_A$ graphs are $R_{2n}=R_n\vee R_n'$, $RL_n$ and $R_n\sqcup C_n$. In the first two cases $X=X_A$ and in the last case $X$
is obtained by connecting $R_n$ to $C_n$ with two edges that have the same endpoints. The action of $W_n$ on
each factor of $R_{2n}=R_n\vee R_n'$ will be as described in (1), except that on at most one factor the $\e_i$ might act trivially. In the action of $W_n$ on $CL_n$, either each $\e_i$
is supported on the $i$th figure-8 graph in the wedge, or else the action is obtained from one with this property by twisting with $\alpha:W_n\to W_n$.
\end{enumerate}
\end{proposition}

\begin{proof} We divide the proof into cases according to the classification in Proposition~\ref{p:An} of $A_n$-graphs.

{\bf Case 1:  $X_A=R_n$ or $C_n$}.  In this case we can apply Lemma~\ref{Wnaction} to the action of $W_n$ on the set of (unoriented) edges of $X$ to conclude that  $X_A$ is invariant and that each $\varepsilon_i$ acts as the identity on the set of edges of $X_A$, and as a fixed involution $\tau$ on $X^A$.  If $X_A=C_n$ and $\varepsilon_1$ inverts an edge, then it must interchange the vertices of $C_n$ and thus invert all of the edges; furthermore, since the $\varepsilon_i$ are all conjugate by the action of $A_n$, they must all do this. 
Thus, regardless of whether the $\varepsilon_i$ invert the edges of $C_n$ or fix them,
 $\varepsilon_1\varepsilon_2$ acts as the identity on $X$, so $X$ does not realize $W_n$ -- a contradiction.
  We conclude that $X_A=R_n$, and we
label the edges so that the action is standard. Each $\varepsilon_i$ acts by flipping some of the petals.  Since all $\varepsilon_i$ are conjugate by elements of $A_n$, they all flip the same number of petals.  If $\varepsilon_i$ flips $a_j$ for some $j\neq i$, it must flip all $a_j$ for $j\neq i$, because  $\varepsilon_i$ commutes with a copy of $A_{n-1}$ which acts transitively on these $a_j$.  It can't flip all (or none) of the petals, since then $\varepsilon_i\varepsilon_j$ would act as the identity.   Therefore $\varepsilon_i$ must flip $e_i$ alone, or else all edges except $e_i$.   

If $m< 2n-2$ this takes care of all possibilities for $X_A$, by Proposition~\ref{p:An}.

{\bf Case 2: $X_A=K(3,n)$.} The full group of isometries of $K(3,n)$ is isomorphic to $S_3\times S_n$, which has order only $6n!$.  This is less than the order of $W_n$, so $K(3,n)$ cannot realize $W_n$.  Adding one or two extra edges to $K(3,n)$ can only reduce the size of the isometry group, so in fact  $X_A$ cannot be isomorphic to $K(3,n)$ for any $m\leq 2n$.

{\bf Case 3:  $X_A=C_n\vee C_n$.}  
Write $X_A=C_n\vee C_n'$ where $C_n'$ is another copy of $C_n$, and $A_n$ acts diagonally.  Applying Lemma~\ref{Wnaction} to the set $\Omega$ consisting of corresponding pairs $\{e_i,e_i'\}$ of edges in $X_A$ and single edges $\{f_i\}\in X^A$,
we conclude that $\varepsilon_i$ either fixes  all $e_j$ or interchanges  each pair $\{e_j,e_j'\}$. But we know that all $\varepsilon_i$ act by the same involution on $X^A$, so this would imply that $\varepsilon_1\varepsilon_2$ acts as the identity on $X=X_A\cup X^A$, contradicting the assumption that the action of $W_n$ is minimal, hence faithful.

{\bf Case 4:  $X_A=C_n\sqcup  C_n$.}  
This cannot be a minimal $W_n$-graph; the proof is identical to Case 3.

{\bf Case 5:  $X_A=\bigvee C_3$.}  Since no edge of $X_A$ can be inverted by an isometry,  $W_n$ acts on the set of $A_n$-orbits of edges. Since $A_n$ acts trivially on this set, the action factors through $W_n/A_n\cong \Z/2\times  \Z/2$.  But any action of $  \Z/2\times  \Z/2$ on a set of $3$ elements has a fixed element.  This means that some $A_n$-orbit is invariant under $W_n$, so $X$ has an invariant forest and is not minimal for $W_n$.

All other cases support a $W_n$ action.  Specifically, we have:

{\bf Case 6: $X_A=R_n\vee C_n$ or $R_n\sqcup C_n$}.   Here, $R_n$ and $C_n$ are each invariant under the full isometry group of $X_A$.  Apply Lemma~\ref{Wnaction} separately to the set of edges in $R_n$ and in $C_n$ to conclude that   $\varepsilon_i$ acts as in Case 1 on $R_n$ and trivially on $C_n$.

{\bf Case 7:  $X_A=C_{2n}$.}   If we write $C_{2n}=C_n\cup C_n'$ with diagonal $A_n$-action, then  Lemma~\ref{Wnaction}  applied to the set $\O$ of corresponding pairs $\{e_i,e_i'\}$ shows that
each $ \varepsilon_j$ acts trivially on the set of such pairs. Arguing as in case 1, we see that $\e_i$ interchanges only $e_i$ and $e_i'$, or else interchanges $e_j$ and $e_j'$ for all $j$ except $j=i$.
In addition, all of the $\e_i$ interchange the vertices of $C_{2n}$, or else fix them.

{\bf Case 8:  $X_A=CL_n$.}  Arguing as in case 1, we see that $\varepsilon_i$ acts by flipping the $i$-th 
loop or else flipping all loops except the $i$-th.
It may also interchange the top and bottom vertices.  (If $\e_i$ did not flip any loops, then $\varepsilon_1\varepsilon_2$ would act as the identity.) If the $\e_i$ do not interchange the vertices, then the
transpositions in $S_n$ must, since otherwise there is an invariant forest.

{\bf Case 9: $X_A=RL_n$.} Again, an argument akin to case 1 shows that (twisting with $\alpha:W_n\to W_n$ if necessary) we may assume that $\varepsilon_i$ is supported on the  $i$-th figure-8 in the wedge.

{\bf Case 10:  $X_A=R_{2n}$}.  We have  $R_{2n}=R_n\vee R_n'$ with $A_n$  acting diagonally. $\varepsilon_i$ acts by flipping $e_i$ and $e_i'$ or flipping all other petals and/or interchanging $e_i$ with $e_i'$.

\end{proof}

\begin{remark} When $n$ is odd, certain of the actions in the preceding proposition may fail to be faithful because of the twisting by $\alpha$: when the action of $W_n$ on $X_A$
factors through $\alpha: W_n\to W_n$, the action of  $\Delta$ on $X^A$ must be non-trivial if the action of $W_n$ on $X$ is to be faithful.
\end{remark}

\subsection{Classification of minimal admissible $G_n$-graphs}

Let $G_n\cong S_{n+1}\times \Z/2$ be the subgroup of $\out$ which is realized as the full automorphism group of the $n$-cage $C_{n+1}$ with the first $n$ edges labelled by the generators $a_1,\dots,a_n$ of $F_n$.  The
$\Z/2$ factor of $G_n$ is generated by $\Delta = \e_1\e_2\dots\e_n$, which interchanges the two vertices of $C_{n+1}$, leaving each unoriented edge invariant.

Let $Y$ be a minimal admissible $G_n$-graph of genus at most $m\le 2n=2(n+1)-2$. Assume $n> 7$.
Since $G_n$ contains $B=A_{n+1}$,  Proposition~\ref{p:An} tells us that
  $Y_B$ is isomorphic to either $C_{n+1}$ or $R_{n+1}$ if $m<2n$, with the additional possibilities   $Y=Y_B=C_{n+1}\vee C_{n+1}$ and $Y=K(3,n+1)$  if $m=2n$.
In fact, this last possibility does not occur, because in any faithful action of $G_n$ on $K(3,n+1)$, the central $\Z/2$ leaves the set of 3 cone points invariant
and hence fixes one of them, so the star of this fixed point is a $G_n$-invariant forest, which shows that the $G_n$-action is not minimal.  

\begin{proposition}\label{p:Gn} If $Y$ is a $G_n$-minimal admissible  graph and $n>7$, then the subgraph $Y_B$  is invariant under all of $G_n$.

If the central element $\Delta\in G_n$ acts on $Y_B$ non-trivially, then it flips all edges  (if $Y_B=R_{n+1}$), interchanges the two vertices  (if  $Y_B=C_{n+1}$) or interchanges the two copies of $C_{n+1}$ (if $Y= C_{n+1}\vee C_{n+1}$) without permuting the edges.  The odd permutations of  $S_{n+1}$  act on $Y_B$  by permuting the edges in the standard way or else
each acts by the permutation  composed with the action of $\Delta$.   

 \end{proposition}

\begin{proof} Since $B=A_{n+1}$ is normal in $G_n$, the action of $G$ preserves the fixed subgraph $Y^B$ of $B$, and hence also preserves the complementary subgraph $Y_B$.  
In particular, $\Delta$  
acts on $Y_B$ by an automorphism that commutes with the $B$-action.  If $Y_B=R_{n+1}$, the only non-trivial graph automorphism which commutes with the $B$-action is the one which flips each petal of the rose without permuting the petals.  If $Y_B=C_{n+1}$, the only non-trivial graph automorphism which commutes with the $B$-action is the one which interchanges the vertices of $C_{n+1}$ without permuting the edges.  If $Y=Y_B=C_{n+1}\vee C_{n+1}$ the only non-trivial graph automorphism which commutes with the (diagonal) $B$-action is the one which interchanges the two copies of $C_{n+1}$.

The statement about the action of odd permutations on $Y_B$ follows from Lemma~\ref{Snaction}.
\end{proof}

\section{Proof of Theorem~\ref{t:nomaps}}

We are now in a position to prove that for $n>8$, any homomorphism from  $\out $ to ${\rm{Out}}(F_m)$ has  image  of order at most two for   $n<m\leq 2n-2$.  If $n$ is even, this can be improved to $m\leq 2n$.  
We first reduce our problem using the following:

\begin{proposition}\label{p:image}
If $n\ge 3$ and $m <2^{n-1}-1$, then
  any homomorphism  $\out \to {\rm{Out}}( F_m)$ which is not injective on
both $W_n$ and on $G_n$ has image of order at most two.  
\end{proposition}
\begin{proof} If  a homomorphism is not injective on $G_n$ then the kernel either consists of the central involution
$\Delta$ or contains $A_{n+1}$. In either case, it follows that the homomorphism is not injective on $W_n$.
In (\cite{BV3}, Proposition C) we proved that any homomorphism
from $\out$ that is not injective on $W_n$ must factor through $\out\to{\rm{PGL}}(n,\Z)$, and by \cite{BridsonFarb}
all homomorphisms ${\rm{PGL}}(n,\Z)\to {\rm{Out}}( F_m)$ have finite image.
 
The kernel of the natural map $\outm\to{\rm{GL}}(m,\Z)$ is torsion-free,
so the image of ${\rm PGL}(n,\Z)$  in $\outm$ maps injectively  to ${\rm{GL}}(m,\Z)$. A non-trivial finite image of ${\rm{PGL}}(n,\Z)$ is
either  just $\Z/2$ (and the map factors through the determinant) or else it contains a  non-trivial finite image of
${\rm{PSL}}(n,\Z)$.
Every finite image of ${\rm{PSL}}(n,\Z)$ is a finite extension of the simple group
${\rm{PSL}}(n,\Z/p)$ for some prime $p$. Lanazuri and Seitz \cite{LSeitz} prove that the 
minimal degree of a complex
representation of ${\rm{PSL}}(n,\Z/p)$  occurs when $p=2$ and is  equal to $m=2^{n-1}-1$.
 Kleidman and Liebeck \cite{KL} prove that no finite extension of ${\rm{PSL}}(n,\Z/p)$ has a
 representation of lesser degree. 
\end{proof}

\begin{remark} For our purposes, it is sufficient to have the above result in
the range $8<n<m\le 2n$ and one can prove this without recourse to \cite{LSeitz}.
Indeed, since an elementary $p$-group in ${\rm{GL}}(m,\Z)$ can be diagonalised in ${\rm{GL}}(m,\C)$, it has rank at most\footnote{using Smith theory one
can improve this to $\lfloor m/2\rfloor$ if $p$ is odd}
 $m$, whereas ${\rm PSL}(n,\Z/p)$
contains an elementary $p$-group of rank
$\lfloor n/2\rfloor^2$, namely the largest unipotent subgroup that one can fit in a square block above the diagonal.  
\end{remark}

We remind the reader that $\det:\out\to\Z/2$ is the composition of the determinant map
${\rm{GL}}(n,\Z)\to\Z/2$ and the natural surjection $\out\to {\rm{GL}}(n,\Z)$. 

\begin{lemma}\label{det} Suppose $m\geq n\geq 3$,
let $\psi_1,\psi_2\in\out$  and let $\phi\colon \out \to {\rm{Out}}( F_m)$  be any homomorphism.  If $\det(\psi_1)=\det(\psi_2)$, then $\det(\phi(\psi_1))=\det(\phi(\psi_2))$  \end{lemma}
 \begin{proof}  For $n\geq 3$, the only surjection from $\out $ to $\Z/2$ is the determinant map.
 \end{proof}

\smallskip

{\em{For the remainder of this section we suppose that $8<n<m$
and that we  have a homomorphism  $$\phi\colon \out \to {\rm{Out}}( F_m)$$ which is injective on $G_n$ and on $W_n$.  We fix a minimal  admissible graph $X$ of genus $m$ realizing  $\phi(W_n)$ and a minimal admissible graph $Y$  of genus $m$ realizing $\phi(G_n)$. }}

\bigskip

Note that  the intersection $G_n\cap W_n$  is isomorphic to $S_n\times \Z/2$, where $S_n$ permutes the generators of $F_n$ and  $\Z/2$ is generated by the automorphism $\Delta$ which inverts all of the generators.  The image of each element in this intersection is realized both as an automorphism of $X$ and as an automorphism of $Y$.   

\begin{proposition}\label{2n}  For $m\leq 2n$, the only possibilities for $X$ and $Y$  are those with  $X_A=R_n$ and $Y_B=C_{n+1}$ or $R_{n+1}$.   
\end{proposition}

\begin{proof}

We first consider the induced action of $\phi(\sigma)$ on $H_1(F_m)$, where $\sigma=(12)(34)\in A_n$.  We calculate the dimension of the (-1)-eigenspace $V_{-1}(\sigma)$ using the action of $\sigma$ on both $X$ and $Y$.
If $Y=R_{n+1}\vee Y^B$ or $C_{n+1}\cup Y^B$, this calculation gives $\dim(V_{-1}(\sigma))=2$, and if $Y=C_{n+1}\vee C_{n+1}$ we have $\dim(V_{-1}(\sigma))=4$. This covers all possibilities for $Y$ by Proposition~\ref{p:Gn}.

Proposition~\ref{p:Wn} lists all possibilities for $W_n$-graphs, for $m\leq 2n$.  Using these, we calculate $\dim(V_{-1}(\sigma))=4$ if  $X_A=C_{2n}, R_n\vee C_n, R_n\sqcup CL_n, R_{2n}$ or $RL_n$, and $\dim(V_{-1}(\sigma))=2$  if $X=R_n\vee X^A$.  

Therefore to prove the proposition we need only eliminate the possibilities that $Y=C_{n+1}\vee C_{n+1}$ (which has rank $2n$) and
\begin{enumerate}
\item $X=C_{2n}\vee S^1$
\item $X=R_n\vee C_n\vee S^1$
\item $X=R_n\sqcup C_n\cup S^1$
\item $X=CL_n\vee S^1$
\item   $X= RL_n, $ and
\item   $X= R_{2n}.$
\end{enumerate}   
We eliminate these possibilities by considering the action of $\Delta$.
We know that $\Delta$ acts on  $Y=C_{n+1}\vee C_{n+1}$ by interchanging the two copies of $C_{n+1}$ and commuting with the $A_{n+1}$-action, so in an appropriate basis for $H_1(F_{2n})$   the matrix of the induced action on $H_1(F_{2n})$  is
 $$
D_Y= \begin{pmatrix}
 0&I \\
 I&0
 \end{pmatrix}.
 $$
This has $(-1)$-eigenspace $V_{-1}(\Delta)$ of dimension exactly $n$, so this must also be true for the action of $\Delta$ on $X$.
If $X=C_{2n}\vee S^1$ (case (1) above), $\Delta$ must therefore interchange the two copies of $C_n$ in $C_{2n}$ and flip the extra $S^1,$ so that the matrix of $\Delta$ is
$$
D_X= \pm  \begin{pmatrix}
0&I_{n-1}&0&0 \\
I_{n-1}&0&0&0 \\
0&0&-1&0 \\
0&0&0&1 \\
 \end{pmatrix}.
 $$
In cases (2)-(5) the fact that $V_{-1}(\Delta)$ has dimension exactly $n$ implies that $\Delta$ acts by flipping exactly $n$ loops in some $A_n$-orbit, and therefore in an appropriate basis  the matrix for $\Delta$ must be
$$
D_X= \pm  \begin{pmatrix}
-I&0 \\
 0&I
 \end{pmatrix}.
 $$
In all of these cases, $D_X$ and $D_Y$ are not conjugate in ${\rm{GL}}(2n,\Z)$.  To see this, note that the sublattice of $\Z^n$ spanned by all eigenvectors has different covolume  for $D_Y$ and $D_X$.

Finally, if $X=R_{2n}=R_n\vee R_n'$ then the same reasoning shows that   the action of $\Delta$ must exchange the two copies of $R_n$. The transposition $(12)$  must act either by sending $e_1\to e_2$ and $e_1' \to e_2'$ or by sending $e_1\to e_2'$ and $e_1' \to e_2$.  In either case $(12)$ acts by a transformation with  determinant +1.   If $n$ is odd, then $\Delta$ acts with determinant $-1$, contradicting Lemma~\ref{det}.  If $n$ is even, then  $\varepsilon_i$ must act by
exchanging the $i$-th edge $e_i$  of $R_n$  with the corresponding edge $e_i'$ of $R_n'$ (possibly flipping them both) and fixing (or flipping) all $e_j$ for $j\neq i$; in any case the induced map on $H_1(F_{2n})$ has determinant -1, again contradicting Lemma~\ref{det}.    
 \end{proof}

\begin{proposition}\label{p:fixed}  For $n<m\leq 2n$, the action of $\Delta$ on $X^A$ must be non-trivial.  
\end{proposition}

\begin{proof}
Suppose that the action of $\Delta$ on $X^A$ is trivial. Then the action of $\Delta$ on $R_n\subset X$ cannot be trivial, and must commute with the action of $A_n$. The only possibility is  that $\Delta$ acts by inverting all of the petals of $R_n$, so that the dimension of  the $(-1)$-eigenspace $V_{-1}(\Delta)$ is exactly equal to  $n$.  

We now calculate the dimension of $V_{-1}(\Delta)$ using $Y$.  If $\Delta$ acts trivially on $Y_B$, then the dimension of $V_{-1}(\Delta)$ is at most the rank of $Y^B$, which is strictly less than $n$.  If $\Delta$ acts non-trivially on $Y_B$ it must invert all edges, since it commutes with the action of $B=A_{n+1}$ on $Y_B$.    If  $Y_B=R_{n+1}$, then it is clear that $V_{-1}(\Delta)$  has dimension at least $n+1$.  This is also true if $Y_B=C_{n+1}$, since $\Delta$ interchanges the vertices of the cage, so must also act non-trivially on $Y^B$.  Thus the computation of $\dim(V_{-1}(\Delta))$ made with $Y$ is inconsistent with the computation made with $X$.  
\end{proof}

\begin{corollary}\label{c:oddDone}   If $n>8$ is even and $n<m\leq 2n$ then  every homomorphism
$\phi\colon \out \to {\rm{Out}}( F_m)$ has image of order at most two.
\end{corollary}

  \begin{proof}  If the image of $\phi$ has order larger than $2$, then by Proposition~\ref{p:image} $\phi$ is injective on $W_n$ and $G_n$, and we have minimal admissible $X$ and $Y$ realizing $\phi(W_n)$ and $\phi(G_n)$ as above.   By Proposition~\ref{p:Wn}, each $\varepsilon_i$ acts by the same involution of  $X^A$; since $n$ is even, this means that $\Delta=\prod \varepsilon_i$ acts trivially on $X^A$, contradicting Proposition~\ref{p:fixed}.
 \end{proof}

\begin{corollary}  If $n>8$ is odd,   then every homomorphism $\out \to {\rm{Out}}( F_{n+1})$ has image of order at most two.
\end{corollary}
\begin{proof}  In this case $X=R_n\vee S^1$, and by Proposition~\ref{p:fixed} we may assume $\Delta$ acts non-trivially on $S^1$.  Thus $\Delta$ acts as $-I$ on $H_1(X)$, which has determinant +1.

According to Proposition~\ref{2n} the possibilities for $Y$ are $Y=R_{n+1}$ or $Y=C_{n+1}\cup e$.   However, the second cannot occur, since the only way to symmetrically add an edge to $C_{n+1}$ is to embed $C_{n+1}$ in $C_{n+2}$; but then $e$ is an invariant forest, so this graph is not minimal for the $G_n$-action.   Therefore $Y=R_{n+1}$, and by Proposition~\ref{p:Gn} the transposition $(12)$ acts on $Y$ with determinant $-1$ (since $n+1$ is even), contradicting Lemma~\ref{det}.
\end{proof}

 If $n$ is odd and $m>n+1$ the above arguments do not work so we employ a different approach.
 
\begin{proposition}\label{p:atmost2}
If $n$ is odd, $n>8$ and $n<m\le 2n-2$ then
the image of any homomorphism $\out \to {\rm{Out}}( F_m)$ has order at most two.
\end{proposition}
 
 \begin{proof}   A theorem of Potapchik and Rapinchuk (\cite{Rapinchuk}, Theorem 3.1)  implies  that for $m\leq 2n-2$  any representation $\out \to {\rm{GL}}(m,\Z)$  factors through the standard representation $\out \to {\rm{GL}}(n,\Z)$.
We apply this fact to the map $\out\to {\rm{GL}}(m,\Z)$ obtained by composing an arbitrary homomorphism $\out\to\outm$
with the natural map $\outm\to {\rm{GL}}(m,\Z)$.

Either the image of
${\rm{SL}}(n,\Z)$ in ${\rm{GL}}(m,\Z)$ under the induced map is finite or else, by super-rigidity, it extends to a representation ${\rm{SL}}(n,\R)\to {\rm{GL}}(m,\R)\subset {\rm{GL}}(m,\C)$. If the image is finite then
as in the proof of Proposition \ref{p:image} it is at most $\Z/2$. In particular, the map from $W_n$ to $\outm$ cannot be injective
because the kernel of $\outm\to {\rm{GL}}(m,\Z)$ is torsion-free. It then follows from the statement of Proposition \ref{p:image} that the image
of our original homomorphism $\out\to\outm$ has order at most 2.

Suppose now that the image of ${\rm{SL}}(n,\Z)$ is infinite and consider its extension $\rho: {\rm{SL}}(n,\R)\to {\rm{GL}}(m,\R)\subset {\rm{GL}}(m,\C)$.
Complete reducibility for  ${\rm{SL}}(n,\R)$ implies that $\rho$ is a sum of irreducible representations (see \cite{FH} page 130).
A calculation with the hook formula shows that the only irreducible representations below dimension $2n$ are the trivial one,
 the standard $n$-dimensional representation and its contragradient. Since we are assuming that the image
of ${\rm{SL}}(n,\Z)$ is infinite, we must have exactly one copy of the standard representation or its contragradient.

Let $\tau\in S_n\subset \out$ be a transposition.
Since $n$ is odd, $\tau\Delta$ has determinant 1. It follows from the above that the $-1$ eigenspace
of $\tau\Delta$ in $H_1(F_m,\C)$ has the same dimension as in the standard representation of ${\rm{SL}}(n,\Z)$,
that is
 $\dim(V_{-1}(\tau\Delta))=n-1$.

We prove that this last equality is impossible by considering the action of $\tau$ on the homology of the graphs $X$ and $Y$. Proposition \ref{2n}
limits the possibilities for $X$ and $Y$, and Propositions \ref{p:Wn}
and \ref{p:fixed} describe the action in each case.

If $\tau$ acts by both permuting the edges of $R_n=X_A\subset X$ and flipping them, then $$\dim(V_{-1}(\tau\Delta))\leq 1 + rank(X^A)< 1+(n-2)=n-1.$$  
Since this is impossible, $\tau$ must act without flipping the edges of $R_n$.  
But then the same calculation shows that $\dim(V_{-1}(\tau))<n-1$.
It follows that $\tau$ must also act without flipping the edges of $Y_B $, since otherwise the dimensions of the $(-1)$-eigenspaces of $\tau$,
as calculated with $X$ and $Y$, would not agree.   But then using $Y$ to compute   $V_{-1}(\tau\Delta)$ gives $\dim(V_{-1}(\tau\Delta))\geq n$, a contradiction.
 \end{proof}

This completes the proof of Theorem \ref{t:nomaps}.

\medskip

In the proof of Corollary \ref{c:oddDone} we invoked Proposition \ref{p:image} to promote the fact that $\Delta\in X^A$ was acting trivially on $X$ to the fact
that the image of $\out$ was finite. Up to that point, we had not used the ambient structure of $\out$ and thus our arguments prove the following  theorem.

\begin{theorem}\label{t:amalgam} Let $W_n = (\Z/2)^n\rtimes S_n$, let $G_n = S_{n+1}\times \Z/2$, and consider the amalgamated free product
$$
P_n = W_n\ast_{(S_n\times \Z/2)}G_n
$$
where the amalgamation identifies the visible $S_n<W_n$  with $S_n<S_{n+1}$ and identifies the $\Z/2$ factor of $G_2$ with the centre of $W_n$ (which therefore
is central in $P_n$).

If $n>8$ is even and $n<m\le 2n$, then the centre of $P_n$ lies in the kernel of every homomorphism $P_n\to \outm$.
\end{theorem}

In the case where $n$ is odd, our proof of Theorem C does not imply
an analogue of Theorem \ref{t:amalgam} because the proof of
Proposition \ref{p:atmost2} relies heavily on the ambient structure
of $\out$ and in particular on its low dimensional representation
theory. That proof begs the question of whether  a closer study of the
representation theory of $\out$, extending Theorem 3.1 of \cite{Rapinchuk} 
and paying particular attention to the
$-1$ eigenspaces of the $\e_i$ and $\Delta$, might allow one to
improve the bound $m\le 2n$ in Theorem C without having to classify all homomorphisms
$W_n ,G_n\to{\rm{Out}}(F_m)$ in the expanded range.
 This idea is pursued by Dawid Kielak in his Oxford
doctoral thesis, cf.~\cite{kielak}.
   
\section{Appendix: Characteristic covers}

The method that we used in
the first part of this paper to construct monomorphisms $\out\hookrightarrow{\rm{Out}}(F_{m})$ was this: we took
a characteristic subgroup of finite index in $F_n$
that contains the commutator subgroup and  split the short exact
sequence $1\to F_n/N\to \aut/N\to \out \to 1$. But in truth it was not this sequence {\em per se} that we split, but rather an
isomorphic sequence involving groups of homotopy equivalences of graphs. The purpose of this appendix is to prove the following
theorem, which
explains why these two splitting problems are equivalent.

\smallskip

\noindent{\em{Notation.}} Let $X$ be a connected
CW complex $X$ with basepoint $x_0$,  let $\he(X)$ be the set of homotopy equivalences $X\to X$, with the compact-open topology, and let $\he_0(X)\subset\he(X)$ be those that fix  $x_0$.  Define
$$\HE(X) = \pi_0(\he(X))\ \  \text{   and   }\ \
\HE_\bullet(X) = \pi_0(\he_0(X)).$$
Thus $\HE(X)$ is the group of homotopy classes of self-homotopy equivalences of $X$, and $\HE_\bullet(X)$ is the group of homotopy classes rel $x_0$ of basepoint-preserving self-homotopy equivalences of $X$. Let $\iota: \HE_\bullet(X)\to \HE(X)$ be the map induced by $\he_0(X)\hookrightarrow\he(X)$.

\smallskip

Given a connected covering space  $p\colon\widehat X\to X$, we define $\fhe(\widehat X)$ to be the set of self-homotopy equivalences $\hat h:\widehat X
\to\widehat X$
that are fibre-preserving, i.e. if $p(\widehat x)=p(\widehat y)$, then $p\widehat h(\widehat x) = p\widehat h(\widehat y)$.
Consider the group
$$\FHE(\widehat X) = \pi_0(\fhe(\widehat X)).
$$

\begin{theorem}\label{diagram}  Let $X$ be a connected CW complex with basepoint $x_0\in X$.  Let $N<\pi = \pi_1(X,x_0)$ be a characteristic subgroup and suppose that the centralizer
$Z_\pi(N)$ is trivial.  Let $p\colon\widehat X\to X$ be the covering corresponding to $N$ and fix $\widehat x_0\in p^{-1}(x_0)$.  Then there is a homomorphism $\delta: \pi\to \HE_\bullet(X)$ and a commutative diagram of groups
$$
\begin{matrix}
&&1&&1\\
&&\downarrow&&\downarrow\\
&&N&=&N\\
&&\downarrow&&\enspace\enspace\downarrow\\ 
1&\to&\pi&\buildrel{\delta}\over\to&\HE_\bullet(X)&\buildrel{\iota}\over\to&\HE(X)&\to&1\\
&&\downarrow&&\enspace\enspace\downarrow{\lambda}&&\| \\
1&\to&\deck&\to&\FHE(\widehat X)&\buildrel{p_*}\over\to&\HE(X)&\to&1\\
&&\downarrow&&\downarrow\\
&&1&&1\\
\end{matrix}
$$
where $\deck\cong \pi/N$ is the group of deck transformations of $p:\widehat X\to X$ and where
$\lambda([h])$ is defined to be the class of the lift of $h$ that fixes $\widehat x_0$.
\end{theorem}

The proof of the above theorem involves little more
than the homotopy extension property, the homotopy lifting property, and some  thought about the role of basepoints. But we found it hard to track down precise
references for the relevant facts (although much of what we need is in \cite{SJ}).
We therefore provide a complete proof. We require three lemmas, the first of
which involves the map $\delta : \pi_1(X,x_0)\to \HE_\bullet(X)$ that is defined as follows.

Let $I=[0,1]$. Given any continuous map $h:X\to X$ and any path $\sigma:I\to X$ from $x_0$ to $h(x_0)$, we apply the homotopy extension principle to
obtain a map $H : X\times [0,1]\to X$ with $H|_{X\times\{0\}}=h$ and $H(x_0,t)=\sigma(1-t)$. (Here $\sigma$ is viewed
as a homotopy of a point.) Define $d(h,\sigma):X\to X$
to be the restriction of $H$ to $X\times\{1\}$; it is thought of as the map obtained from
$h$ by ``dragging $h(x_0)$ back to $x_0$ along $\sigma$".
Note that $h$ is freely homotopic to $d(h,\sigma)$. Note too that a further application of homotopy extension shows that
a different choice of homotopy $H'$ would lead to a map $d'(h,\sigma)$ that is homotopic to $d(h,\sigma)$ rel $x_0$.

If $\sigma\simeq \sigma'$ rel endpoints, then by a further application of homotopy
extension we see that $d(h,\sigma)\simeq d(h,\sigma')$ rel $x_0$. In particular, if $\sigma$ is a loop based at $x_0$,
then the based homotopy class of $d({\rm{id}}_X, \sigma)$ depends only on $[\sigma]\in\pi_1(X,x_0)$.
Thus we obtain a well-defined map $\delta : \pi_1(X,x_0)\to \HE_\bullet(X)$ by defining $\delta([\sigma]):=[d({\rm{id}}_X, \sigma)]$.
And because we dragged backwards along $\sigma$
in the definition of $d(h,\sigma)$, this is a homomorphism.

\begin{lemma}\label{l:exactone} Let $\pi=\pi_1(X,x_0)$ and suppose the center $Z(\pi)$ is trivial.  Then the following sequence is exact:
$$1 \to \pi \buildrel{\delta}\over\to \HE_\bullet(X)\buildrel{\iota}\over\to \HE(X)\to 1.$$
\end{lemma}

\begin{proof}
Given $h\in \HE(X)$, we choose a path $\sigma$ from $x_0$ to $h(x_0)$.
By construction, $d(h,\sigma)$ fixes $x_0$ and is freely homotopic to $h$.  Thus $\iota$ is surjective.   

To see that  ${\rm{im}}(\delta)\subset {\rm{ker}}(\iota)$, note that  the homotopy used to define $\delta([\sigma])$ gives a (free) homotopy
from $\delta([\sigma])$ to ${\rm{id}}_X$.  To establish the opposite inclusion, we
fix $h\in {\rm{ker}}(\iota)$ and choose a homotopy  $G$ of $h$ to the identity. Let $\sigma(t)=G(x_0,1-t)$.  Then, by definition,
$\delta([\sigma]) = h$.

To see that $\delta$ is injective, fix  a loop $\gamma$ and
suppose   that $\delta([\gamma])$ is trivial, i.e. that there is a basepoint preserving homotopy from $d({\rm{id}}_X,\gamma)$
to ${\rm{id}}_X$.  By combining
 this  homotopy  with the homotopy $H$ used to define $d({\rm{id}}_X,\gamma)$, we get a homotopy $F:X\times [-1,1]\to X$
from ${\rm{id}}_X$ to itself with $F|_{\{x_0\}\times [0,1]}=\gamma$ and $F|_{\{x_0\}\times [-1,0]}$ a 
constant path at $x_0$; let $\gamma'
:[-1,1]\to X$ be this
reparameterisation of $\gamma$.
Given any
loop $\tau$ based at $x_0$, the map $I\times [-1,1]\to X$ defined by $(s,t)\mapsto F(\tau(s),t)$ restricts on the top and bottom of the square to $\tau$ and
on the two sides to $\gamma'$. Thus  $[\tau]$ and $[\gamma']=[\gamma]$ commute in $\pi_1(X,x_0)$.
Since $\tau$ is arbitrary, we conclude that  $[\gamma]$ is in the centre of $\pi$, which is trivial  by hypothesis.  
\end{proof}

Now let  $p\colon\widehat X\to X$ be  a connected normal covering space, fix  $\widehat x_0$
with $p(\widehat x_0) = x_0$, and
let $N=p_*\pi_1(\widehat X,\widehat x_0)$.  
If $N$ is a characteristic subgroup of $\pi_1(X,x_0)$, then we say that the covering is
{\it characteristic}.

\begin{lemma}\label{l:exacttwo}  Let $p\colon\widehat X\to X$ be a characteristic covering space, with group of deck transformations $\deck
= \pi_1(X,x_0)/N$, and assume that the centralizer of $N$ in $\pi_1(X,x_0)$ is trivial.  Then the following sequence is exact:
$$1\to   \deck \buildrel{}\over\to  \FHE(\widehat X) \buildrel{p_*}\over\to \HE(X)\to 1.$$
\end{lemma}
\begin{proof}  
Every $h\in\he(X)$ lifts to a self-homotopy equivalence $\widehat h\in \fhe(\widehat X)$, because $N$ is characteristic and therefore $h_*(N)=N$.
Thus $p_*: \FHE(\widehat X) \to \HE(X)$ is surjective.

Let $\pi= \pi_1(X,x_0)$. The map $\pi\to {\rm{Aut}}(N)$ defined by conjugation is injective, because we have assumed that the
centralizer  of $N$ in $\pi$ is trivial. It follows that the induced map ${\rm{ad}}: \pi/N\to {\rm{Out}}(N)$ is also injective. But ${\rm{ad}}$
is the natural map from $\deck=\pi/N$ to ${\rm{Out}}(\pi_1(\widehat X, \widehat x_0)) = {\rm{Out}}(N)$ (where the identifications
are given by path-lifting). And $\deck\to {\rm{Out}}(\pi_1(\widehat X, \widehat x_0))$ extends to $\HE(\widehat X)\to {\rm{Out}}( N)$.
Therefore $\deck\to\FHE(\widehat X)$ is injective. Moreover, it is clear that the image of this map
is contained  in ${\rm ker}(p_*)$; this just says that deck transformations project to the identity.  Conversely, if $p_* \widehat h$ is homotopic to the identity,  the homotopy can be lifted to a fiber-preserving homotopy of $\widehat h$ which covers the identity; but the only maps which cover the identity are deck transformations.
\end{proof}

We are studying the normal covering $p:\widehat X\to X$. Path-lifting at the basepoint $\widehat x_0\in p^{-1}(x_0)$ gives the standard identification $\deck\cong \pi_1(X,x_0)/N$; we write $\dd(\gamma)$ for the deck transformation determined by $\gamma$, and we write $[\dd(\gamma)]$ for its image in $\FHE(\widehat X)$. The homomorphism $\lambda : \HE_{\bullet}(X)\to \FHE(\widehat X)$ was defined in  the statement of  Theorem \ref{diagram}: it sends $[h]$  to  the fibre-preserving homotopy class of the lift of $h$ that fixes $\widehat x_0\in\widehat X$. The homomorphism $\delta : \pi_1(X,x_0)\to \HE_{\bullet}(X)$ was defined prior to
Lemma~\ref{l:exactone}.

\begin{lemma}\label{l:commSq} For all $\gamma\in\pi_1(X,x_0)$ we have $\lambda(\delta(\gamma)) = [\dd(\gamma)]$.
\end{lemma}

\begin{proof} Fix a loop $\sigma [0,1]\to X$ with $[\sigma]=\gamma$ in $\pi_1(X, x_0)$. The construction of $\delta(\gamma)$ involves a homotopy $H:X\times [0,1]\to X$
with $H(x_0,t) = \sigma(1-t)$ and $H|_{X\times\{0\}}={\rm{id}}_X$ while $h_1:=H|_{X\times\{1\}} \in \delta(\gamma)$. By definition, $\lambda(\delta(\gamma))$ is the
fibre-preserving homotopy class of the lift $\widehat h_1$ of $h_1$ that fixes $\widehat x_0$. Now $H$ lifts to a fibre-preserving homotopy $\widehat H: \widehat X\times [0,1]
\to \widehat X$ with $\widehat H_{\widehat X\times\{1\}} = \widehat h_1$ and $\widehat H(\widehat x_0,t) = \widehat\sigma(1-t)$, where $\widehat\sigma : [0,1]\to \widehat X$
is the lift of $\sigma$ with $\widehat\sigma(0)=\widehat x_0$. Thus $\widehat H|_{\widehat X\times\{0\}}$ is the lift of ${\rm{id}}_X$ that sends $\widehat x_0$ to
$\widehat\sigma(1)$. By definition, this lift of ${\rm{id}}_X$ is $\dd([\sigma])=\dd(\gamma)$. Therefore
$\widehat H$ is a fibre-preserving homotopy from $\widehat h_1$
to $\dd(\gamma)$, showing that $\lambda(\delta(\gamma)) = [\dd(\gamma)]$.
\end{proof}

\noindent{\bf Proof of Theorem~\ref{diagram}}
Lemmas~\ref{l:exactone} and \ref{l:exacttwo} tell us that the rows of the diagram are exact. It is clear from the definitions that $p_*\lambda = \iota$,
and Lemma~\ref{l:commSq} tells us that the square beneath $\pi\overset{\delta}\to \HE_{\bullet}(X)$ commutes. Thus the diagram is commutative.
With commutativity in hand, an elementary diagram chase proves that second column is exact.
\qed

\begin{corollary}\label{c:coroll}  Let $X$ be a $K(\pi,1)$ space and let $p\colon\widehat X\to X$
be a covering space with $N=p_*\pi_1(\widehat X)$ characteristic in $\pi$. If $Z_\pi(N)$ is trivial,
 then  the following diagram of groups is commutative and the vertical maps are isomorphisms:

$$
\begin{matrix}

 1&\to &\pi/N&\to &{\rm{Aut}}(\pi)/N&\to &{\rm{Out}}(\pi)\to 1\\
&&\downarrow&&\downarrow &&\downarrow\\
 1&\to &\deck&\to &\FHE(\widehat X)&\to &\HE(X)\to 1\\

\end{matrix}
$$
where
$\pi/N\to {\rm{Aut}}(\pi)/N $ is the map induced by the action of $\pi$ on itself by inner automorphisms.
\end{corollary}

\begin{proof} When $X$ is a $K(\pi,1)$, the natural maps $\HE_\bullet(X)\to {\rm Aut}(\pi)$ and $\HE(X)\to{\rm Out}(\pi)$
are isomorphisms, which we use to identify these groups. By definition $\delta(\gamma)$ is the class of the homotopy equivalence that drags the basepoint of
$X$ backwards around the loop $\gamma$, and therefore the map that it induces on $\pi=\pi_1(X, x_0)$ is the inner automorphism
by $\gamma$. Thus, with the above identifications, by factoring out $N$ from the top row of the diagram in Theorem \ref{diagram}
we obtain the top row of the diagram displayed in the statement of the corollary.
\end{proof}

\bibliographystyle{siam}

\end{document}